\documentclass{amsart}
\usepackage{amssymb,euscript,amsmath, mathrsfs, stmaryrd, textcomp}
\usepackage{comment}
\usepackage[pdftex]{graphicx}
\usepackage[pdftex]{color}
\definecolor{navy}{rgb}{0,0,.5}
\definecolor{darkred}{rgb}{1,0,0}
\usepackage{pictexwd,dcpic}

\input xy
\xyoption{all}
\CompileMatrices

\def\G{{\mathbf G}}

\def\CC{{\mathbb C}}

\def\RR{{\mathbb R}}
\def\ZZ{{\mathbb Z}}
\def\QQ{{\mathbb Q}}
\def\AA{{\mathbb A}}
\def\PP{{\mathbb P}}
\def\TT{{\mathbb T}}
\def\GG{{\mathbb G}}

\def\Puis{{\mathbb{C} \{\!\{ t \}\!\}}}

\def\Spec{\operatorname{Spec}}

\def\inn{\operatorname{in}}

\def\Trop{\operatorname{Trop}}
\def\trop{\operatorname{trop}}

\def\Hilb{\operatorname{Hilb}}
\def\oPic{\operatorname{\overline{Pic}}}
\def\Jac{\operatorname{Jac}}
\def\oJac{\operatorname{\mathscr{J}}}

\def\an{\mathrm{an}}
\def\cC{{\mathscr C}}

\def\cU{{\mathcal U}}

\def\cX{{\mathcal X}}
\def\cY{{\mathcal Y}}
\def\cZ{{\mathcal Z}}

\def\<{{\langle}}
\def\>{{\rangle}}

\DeclareMathOperator{\eu}{eu}

\DeclareMathOperator{\val}{val}

\newcommand{\ignore}[1]{}

\numberwithin{equation}{section}
\newtheorem{thm}[equation]{Theorem}
\newtheorem{prop}[equation]{Proposition}
\newtheorem{lem}[equation]{Lemma}
\newtheorem{cor}[equation]{Corollary}
\newtheorem{conj}[equation]{Conjecture}

\newtheorem*{thm*}{Theorem}
\newtheorem{thm2}{Theorem}[section]

\theoremstyle{definition}
\newtheorem{defn}[equation]{Definition}

\newtheorem{ex}[equation]{Example}
\theoremstyle{remark}

\newtheorem{rem}[equation]{Remark}

\newtheorem{examples}[thm2]{Examples}

\newenvironment{conjtag}[2]{\vspace{5pt} \noindent {\bf Conjecture #1.} {\it #2}}{}

\newcommand{\spe}{\mathrm{sp}}
\newcommand{\VF}{\mathrm{VF}}
\newcommand{\Var}{\mathrm{Var}}
\newcommand{\Vol}{\mathrm{Vol}}
\newcommand{\RES}{\mathrm{RES}}

\newcommand{\rec}{\mathrm{rec}}

\def\Q{\mathbb{Q}}
\def\LL{{\mathbb L}}

\begin{document}
\title{Tropical refined curve counting via motivic integration}
\author{Johannes Nicaise}
\author{Sam Payne}
\author{Franziska Schroeter}
\address{Johannes Nicaise, Imperial College,
Department of Mathematics, South Kensington Campus,
London SW72AZ, UK, and KU Leuven, Department of Mathematics, Celestijnenlaan 200B, 3001 Heverlee, Belgium} \email{j.nicaise@imperial.ac.uk}
\address{Sam Payne, Mathematics Department, Yale University, New Haven, CT 06511,USA}
\email{sam.payne@yale.edu}
\address{Franziska Schroeter, Fachbereich Mathematik (AD), Universit\"at Hamburg, Bundesstr. 55, 20146 Hamburg, Germany}
\email{franziska.schroeter@uni-hamburg.de}

\begin{abstract}
  We propose a geometric interpretation of Block and G\"{o}ttsche's refined tropical curve counting invariants in terms of virtual $\chi_{-y}$ specializations of motivic measures of semialgebraic sets in  relative Hilbert schemes.  We prove that this interpretation is correct for linear series of genus 1, and in arbitrary genus after specializing from $\chi_{-y}$-genus to Euler characteristic.
\end{abstract}

\maketitle

\vspace{-20pt}

\section{Introduction} \label{sec:intro}

Geometers have developed an array of sophisticated techniques for counting special curves in linear series, from degenerations and cobordism to stability conditions on derived categories.  Our paper examines the relationship between two relatively naive approaches to curve counting in toric surfaces, one using Euler characteristics of relative compactified Jacobians and Hilbert schemes, and the other using tropical geometry.  We are especially motivated by ``refined" versions of these curve counts, in which Euler characteristics are replaced with $\chi_{-y}$-genera\footnote{The $\chi_{-y}$-genus of a smooth projective complex variety $X$ is the polynomial $\chi_{-y}(X) = \sum_{q} (-1)^{q}\chi(X,\Omega^q_X)y^q$. This definition extends uniquely to an invariant $\chi_{-y}$ for arbitrary complex varieties that is additive with respect to finite partitions into subvarieties, and which specializes to the Euler characteristic by setting $y=1$. Some authors instead use a ``normalized" $\chi_{-y}$-genus, given by $y^{-\dim(X)/2} \chi_{-y}(X)$, which is a Laurent polynomial in $y^{1/2}$ that is symmetric under $y^{1/2} \mapsto y^{-1/2}$. Beware that this normalized $\chi_{-y}$-genus is not additive and, hence, does not extend to arbitrary varieties in any direct way.  See \S\ref{ss:chiy} for further details.} and combinatorially defined tropical multiplicities are replaced with polynomials, or Laurent polynomials, in a variable $y$ that specialize to the ordinary multiplicities by setting $y =1$.  Our aim is to give a geometric interpretation for the Block--G\"ottsche refined tropical multiplicities from \cite{BG14}, which, on a few specific toric surfaces and conjecturally much more generally, can be used to express G\"ottsche and Shende's refined curve counting invariants from \cite{GS14} as a sum over tropical curves. We expect that our geometric interpretation will be useful to prove the conjectured refined correspondence theorem; we will outline the strategy in Section \ref{ss:conj}.

Our approach is inspired by the growing web of connections between tropical geometry, Berkovich spaces, and motivic integration.  We associate a semialgebraic set in the relative compactified Jacobian or relative Hilbert scheme of points to each tropical curve, and relate combinatorially defined tropical multiplicities to motivic invariants of these semialgebraic sets, using the theory of motivic integration of Hrushovski and Kazhdan \cite{HK06}, together with recent results on the $\ell$-adic cohomology of locally closed semialgebraic sets \cite{Mar14, HL15}.

\subsection{Block and G\"ottsche's refined tropical multiplicities}\label{ss:BG}

Let $Y(\Delta)$ be the projective toric surface associated to a lattice polygon $\Delta$, and denote by $|L(\Delta)|$ the corresponding complete linear series.  Suppose that $\Delta$ has $g$ interior lattice points and $n+1$ total lattice points.  Then the linear series $|L(\Delta)|$ has dimension $n$, its general member is a smooth projective curve of genus $g$, and the locus of irreducible $\delta$-nodal curves in $|L(\Delta)|$ has codimension $\delta$, for $0 \leq \delta \leq g$.  We write $n^{\Delta, \delta}$ for the corresponding \emph{toric Severi degree}, the number of integral $\delta$-nodal curves in $|L(\Delta)|$ that pass through $n - \delta$ points in general position.

These toric Severi degrees can be computed tropically, as the number of plane tropical curves $\Gamma$ of degree $\Delta$ and genus $g - \delta$ passing through $n - \delta$ points in general position, counted with combinatorially defined integer multiplicities $n(\Gamma)$.  These ordinary tropical multiplicities may be interpreted geometrically through the tropical--nonarchimedean correspondence theorems \cite{NS06, Nis09, Gro11, Tyo12}, as follows.  Let $K=\Puis$ be the field of Puiseux series, fix $n - \delta$ algebraic points over $K$ whose tropicalizations are the given $n - \delta$ tropical points in general position, and let $|L| \subset |L(\Delta)|$ be the linear series of curves over $K$ passing through these points.  Then $n(\Gamma)$ is the number of $\delta$-nodal curves in $|L|$ whose tropicalization is $\Gamma$.

The ordinary tropical multiplicity $n(\Gamma)$ can also be expressed combinatorially as a product over the trivalent vertices in the tropical curve, with positive integer factors.
Recently, Block and G\"ottsche have introduced a \emph{refined tropical multiplicity} $N(\Gamma)$, which is a Laurent polynomial in a single variable $y$, and which is expected to form the tropical counterpart of the refined curve counting invariants of G\"ottsche and Shende \cite{GS14}. It may be expressed similarly to the ordinary tropical multiplicity, as a product over trivalent vertices, and specializes to $n(\Gamma)$ by setting $y = 1$ \cite{BG14}.  We will recall the precise definition in \S\ref{sec:tropicalcounting}. As mentioned above, our goal is to give a geometric interpretation for these combinatorial invariants $N(\Gamma)$, generalizing the nonarchimedean-tropical correspondence theorems for ordinary multiplicities. Our approach relies on a suitable generalization of the refined curve counts of G\"ottsche and Shende to families of curves over a semialgebraic base. This requires us to define the $\chi_{-y}$-genus of a semialgebraic set, for which we use the theory of motivic integration developed by Hrushovski
 and Kazhdan \cite{HK06}. Other interpretations of the Block--G\"ottsche invariants in terms of wall-crossings and geometry of real curves are presented in \cite{FS15} and \cite{Mik16}, respectively; it would be interesting to understand the precise relation with the construction that we present here.

\subsection{The refined curve counts of G\"ottsche and Shende}
We begin with rational curve counting invariants (the case where $\delta = g$), following the well-known approach via Euler characteristics of relative compactified Jacobians.  Recall that, if $\cC \rightarrow U$ is a family of integral curves on a smooth surface with finitely many rational fibers, and if all rational fibers have only nodal singularities, then the number of rational fibers is equal to $\eu(\oJac(\cC))$, the Euler characteristic of the relative compactified Jacobian \cite{YZ96, Bea99, FGS99}.  This follows easily from the fact that the compactified Jacobian of each rational fiber has Euler characteristic 1, and the compactified Jacobian of each of the other fibers has Euler characteristic 0.

There is an analogous approach to counting $\delta$-nodal curves via Euler characteristics when $\delta$ is not necessarily equal to $g$, using the relative Hilbert scheme of points rather than the relative compactified Jacobian. Given family of curves $\cC\to U$ and a nonnegative integer $i$, let $\Hilb^i(\cC)$ be the associated relative Hilbert scheme of points, parameterizing families of zero-dimensional subschemes of length $i$ in the fibers. There exists a unique sequence of integers $(n_r(\cC))_{r\geq 0}$ satisfying
\[
q^{1-g} \sum_{i = 0}^\infty \, \eu(\Hilb^i(\cC)) \, q^i = \sum_{r = 0}^\infty n_r(\cC)\, q^{r+1-g}(1-q)^{2g-2r-2}.
\]
It was shown in \cite{GV98, PT10, KST11} that if $\cC \rightarrow U$ is a family of reduced curves of arithmetic genus $g$ on a smooth surface with finitely many $\delta$-nodal fibers, in which all other fibers have geometric genus greater than $g - \delta$, then $n_r(\cC)$ vanishes for $r>\delta$, and the number of $\delta$-nodal curves is equal to $n_{\delta} (\cC)$.\footnote{Be aware that in the three references cited above, the indices are permuted by $r \mapsto g-r$.  We follow the notation of \cite{GS14}, in which $n_r$ is, roughly speaking, a virtual count of curves of cogenus $r$.} The geometric hypotheses on the fibers of the family are satisfied in many natural situations, e.g. for $\delta$-dimensional families of integral curves satisfying mild positivity conditions on a smooth rational surface \cite[Proposition~2.1]{Har86}, and for a general linear series of dimension $\delta$ in the complete linear series of a $\delta$-very ample line bundle on any smooth projective surface \cite[Proposition~2.1]{KST11}.

 In \cite{GS14}, G\"ottsche and Shende have proposed refined curve counting invariants using a similar generating series approach, replacing Euler characteristics with $\chi_{-y}$-genera.  They observe that the generating series
\[
q^{1-g} \sum_{i=0}^\infty \, \chi_{-y} (\Hilb^i(\cC)) \, q^i
\]
can be expressed uniquely as a sum
\[
\sum_{r = 0}^\infty N_r(\cC) \, q^{r+1-g}(1-q)^{g-r-1}(1-qy)^{g-r-1},
\]
where each coefficient $N_r(\cC)$ is a polynomial in $y$ that specializes to $n_r(\cC)$ by setting $y = 1$. G\"ottsche and Shende show that $N_r(\cC)$ vanishes for $r > g$ when $\cC\to U$ is a family of integral curves on a smooth surface and $\Hilb^i(\cC)$ is smooth for all $i$ \cite[Proposition~42]{GS14}. They conjecture that $N_r(\cC)$ also vanishes for $r > \delta$ under the additional assumption that $U \cong \PP^\delta$ and $\cC$ is the universal family of a linear system of curves on a smooth projective surface \cite[Conjecture~5]{GS14}.   The main complication in working with $\chi_{-y}$ instead of the Euler characteristic is that the $\chi_{-y}$-genus of a family cannot be computed by integrating over the base, because of the monodromy action on the induced variation of Hodge stuctures. This is apparent, for instance, in Example \ref{ex:cubic}, below.

\subsection{Linear series on toric surfaces}\label{ss:conj}

We return to the set-up of \S\ref{ss:BG}. Let $\Delta$ be a lattice polygon in $\RR^2$ with $n+1$ lattice points and $g$ interior lattice points. We denote by $(Y(\Delta),L(\Delta))$ the associated polarized toric surface over the field of Puiseux series $\Puis$. The complete linear series $|L(\Delta)|$ has dimension $n$, and its general member is a smooth projective curve of genus $g$. We fix an integer $\delta$, with $0\leq \delta\leq g$. Let $S$ be a set of $n-\delta$ closed points on the dense torus in $Y(\Delta)$ whose tropicalization $\trop(S)\subset \RR^2$ is in general position, and let $|L| \subset |L(\Delta)|$ be the linear series of curves passing through $S$. We denote by $\cC\to |L|$ the universal curve over $|L|\cong \PP^{\delta}$.

First, assume that $\delta=g$, and let $\Gamma\subset \RR^2$ be a rational tropical curve of degree $\Delta$ through the points of $\trop(S)$. As we will explain in \S\ref{sec:fixtrop}, the curves in $|L|$  whose tropicalizations are equal to $\Gamma$ form a semialgebraic set $|L|_\Gamma$ in $|L|$.  Let $\cC_\Gamma$ be the universal curve over $|L|_\Gamma$ and let $\oJac(\cC_\Gamma)$ be the relative compactified Jacobian.  In other words, $\cC_\Gamma$ is the preimage of $|L|_\Gamma$ in the universal curve $\cC \rightarrow |L|$ and, similarly, $\oJac(\cC_\Gamma)$ is the preimage of $|L|_\Gamma$ in the relative compactified Jacobian $\oJac(\cC\times_{|L|}U) \rightarrow U$, where $U \subset |L|$ is the open subvariety parameterizing integral curves  (in \S\ref{sec:integral}, we show that $|L|_\Gamma$ is contained in $U$). We define the Euler characteristic and $\chi_{-y}$-genus of these semialgebraic sets as specializations of the Hrushovski--Kazhdan motivic volume from \cite{HK06} (see Definitions \ref{def:chiy} and \ref{def:eu}).  We conjecture the following geometric interpretation of the Block-G\"ottsche multiplicity $N(\Gamma)$ for counting rational curves.

\begin{conj} \label{conj:rational}
Assume that $\delta=g$ and let $\Gamma\subset \RR^2$ be a rational tropical curve of degree $\Delta$ through the points of $\trop(S)$. Then the Block--G\"ottsche refined tropical multiplicity $N(\Gamma)$ is equal to $y^{-g} \chi_{-y}(\oJac(\cC_\Gamma))$.
\end{conj}

As evidence in favor of this conjecture, we prove that it is correct when $g = 1$ (see Theorem \ref{thm:genus1}). The proof is based on a tropical formula for the Hrushovski--Kazhdan motivic volume of sch{\"o}n subvarieties of tori, which is of indepdendent interest.  See \S\ref{sec:motivic}. We will also prove that our conjecture holds after specialization to ordinary Euler characteristic: if $\delta = g$ then $ n(\Gamma)=\eu(\oJac(\cC_\Gamma))$ (see Theorem \ref{thm:toricSeveri}). This gives a new geometric interpretation of the classical tropical multiplicity $n(\Gamma)$. Our argument uses a comparison result for the motivic Euler characteristic of a semialgebraic set and Berkovich's $\ell$-adic cohomology for nonarchimedean analytic spaces (Proposition \ref{prop:analytic}); this allows us to prove that the motivic Euler characteristic satisfies some standard cohomological properties (see in particular Corollary \ref{cor:eufibers}).

\smallskip

We now drop the assumption that $\delta = g$ and state analogous conjectures and results for counting nodal curves of arbitrary genus.   Let $\Gamma \subset \RR^2$ be a tropical curve of genus $g-\delta$ and degree $\Delta$ through the points of $\trop(S)$.  Adapting the approach of G\"ottsche and Shende to families of curves over a semialgebraic base, we define polynomials $N_r(\cC_{\Gamma})$ in $\ZZ[y]$, for $r\geq 0$, by means of the equality
\[
q^{1-g} \sum_{i=0}^\infty \, \chi_{-y} (\Hilb^i(\cC_\Gamma)) \, q^i
=
\sum_{r = 0}^\infty N_r(\cC_\Gamma) \, q^{r+1-g}(1-q)^{g-r-1}(1-qy)^{g-r-1}.
\]
Here, the relevant geometric object $\Hilb^i(\cC_\Gamma)$ is the preimage of the semialgebraic set $|L|_\Gamma \subset |L|$ in the relative Hilbert scheme $\Hilb^i(\cC) \rightarrow |L|$. Note that the $r$th term in the right hand side is a Laurent series in $q$ with leading exponent $r+1-g$, leading to a recursive formula for the $N_r(\cC_\Gamma)$ with a unique solution.

\begin{conj} \label{conj:chiydelta}
For any value of $\delta$ in $\{0,\ldots,g\}$, let $\Gamma\subset \RR^2$ be a tropical curve of genus $g-\delta$ and degree $\Delta$ through the points of $\trop(S)$.  Then the Block--G\"ottsche refined tropical multiplicity $N(\Gamma)$ is equal to $y^{-\delta} N_{\delta}(\cC_\Gamma)$.
\end{conj}

\noindent We prove that this conjecture, also, is correct for $g=1$ (Theorem \ref{thm:genus1}), for $\delta \leq 1$ (Theorem~\ref{thm:delta1}), and after specialization to $y = 1$ (Theorem \ref{thm:toricSeveri}). Moreover, we will show that Conjecture \ref{conj:chiydelta} implies Conjecture \ref{conj:rational}; see Corollary \ref{cor:imply} for a more precise statement. Note that Conjecture~\ref{conj:chiydelta} also implies that the evaluation of $y^{-\delta} N_\delta(\cC_\Gamma)$ at $y = -1$ is equal to the tropical Welschinger invariant of $\Gamma$, as studied in \cite{IKS09}.

  Conjectures \ref{conj:rational} and \ref{conj:chiydelta} provide a strategy to prove the correspondence conjecture between tropical and geometric refined curve counts \cite[2.12]{BG14}. The correspondence conjecture states that the tropical and geometric counts agree under a suitable positivity condition on the line bundle $L(\Delta)$ (namely, $\delta$-very ampleness).
 Assume that $g=\delta$ and that Conjecture \ref{conj:rational} holds. In order to prove the correspondence conjecture, it would be sufficient to show that the locus of curves in $\mathscr{C}$ whose tropicalizations are not rational, has $\chi_{-y}$-genus equal to zero. The correspondence conjecture then follows from the additivity of the $\chi_{-y}$-genus on finite semi-algebraic partitions. A similar strategy can be applied to the case where $\delta$ is arbitrary, starting from Conjecture \ref{conj:chiydelta} and proving the vanishing of the  contribution of the curves whose tropicalization has genus larger than $g-\delta$.

We conclude the introduction with an example illustrating Conjecture~\ref{conj:rational} in the genus one case.

\begin{ex}\label{ex:cubic}
Let $\cC \rightarrow \PP^1$ be the pencil of cubics through 8 general points in $\PP^2$.  This pencil contains 12 rational fibers.  This can be seen by noting that the total space $\cC$ is isomorphic to the blowup of $\PP^2$ at the 9 basepoints of the pencil, and hence has Euler characteristic 12.  It is straightforward to check that cuspidal curves, reducible curves, and non-reduced curves all have codimension greater than 1 in the space of all cubics, and hence a general pencil contains only smooth curves of genus 1 and nodal rational curves.  Since genus 1 curves have Euler characteristic 0 and rational nodal cubics have Euler characteristic 1, it follows that there must be exactly 12 rational fibers.

This count can also be performed tropically.  There are two possibilities for the collection of tropical rational curves through 8 general points in $\RR^2$.  Either there are 9 such curves, of which 8 contain a loop and 1 contains no loop but has a bounded edge of weight 2, or there are 10 such curves, of which 9 contain a loop and 1 contains no loop, but has a vertex whose outgoing edge directions span a sublattice of index 3.  Each tropical curve $\Gamma$ containing a loop counts with tropical multiplicity $n(\Gamma) = 1$, the curve $\Gamma'$ with a bounded edge of weight 2 counts with multiplicity $n(\Gamma') = 4$, and the curve $\Gamma''$ with a vertex of multiplicity 3 counts with multiplicity $n(\Gamma'') = 3$.  The refined tropical multiplicities, as defined combinatorially by Block and G\"ottsche, are
\[
N(\Gamma) = 1,  \mbox{ \ \ } N(\Gamma') = y^{-1} + 2 + y, \mbox{ \ \ and \ \ } N(\Gamma'') = y^{-1} + 1 + y,
\]
respectively.  The following figure illustrates examples of tropical rational curves of degree 3 in $\PP^2$ with a loop and a node, an edge of weight 2, and a vertex of multiplicity 3; the node, the edge of weight 2, and the vertex of multiplicity 3 are marked in blue.

\bigskip

\begin {center} \scalebox{0.7}{\begin{picture}(0,0)%
\includegraphics{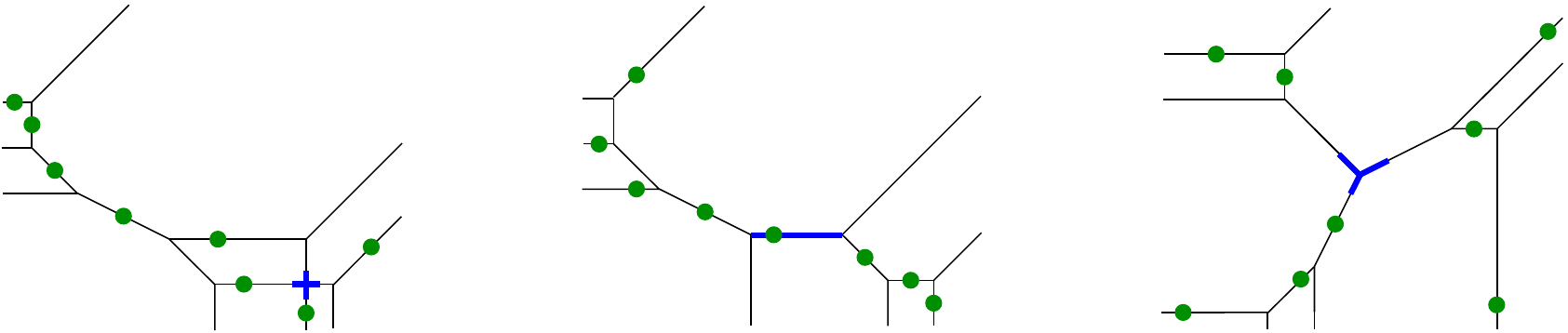}%
\end{picture}%
\setlength{\unitlength}{3947sp}%
\begingroup\makeatletter\ifx\SetFigFont\undefined%
\gdef\SetFigFont#1#2#3#4#5{%
  \reset@font\fontsize{#1}{#2pt}%
  \fontfamily{#3}\fontseries{#4}\fontshape{#5}%
  \selectfont}%
\fi\endgroup%
\begin{picture}(8073,1703)(8846,-11484)
\put(10072,-10293){\makebox(0,0)[lb]{\smash{{\SetFigFont{14}{16.8}{\familydefault}{\mddefault}{\updefault}{\color[rgb]{0,0,0}$\Gamma$}%
}}}}
\put(12901,-10261){\makebox(0,0)[lb]{\smash{{\SetFigFont{14}{16.8}{\familydefault}{\mddefault}{\updefault}{\color[rgb]{0,0,0}$\Gamma'$}%
}}}}
\put(15901,-10261){\makebox(0,0)[lb]{\smash{{\SetFigFont{14}{16.8}{\familydefault}{\mddefault}{\updefault}{\color[rgb]{0,0,0}$\Gamma''$}%
}}}}
\put(12901,-10936){\makebox(0,0)[lb]{\smash{{\SetFigFont{14}{16.8}{\familydefault}{\mddefault}{\updefault}{\color[rgb]{0,0,1}$2$}%
}}}}
\end{picture}%
} \end {center}

\bigskip

Let us briefly explain how we confirm Conjecture \ref{conj:rational} in these cases, postponing the details of the computations to \S\ref{sec:genus1}.

Each of the tropical curves $\Gamma$ with a loop also has a node where the images of two different edges in the rational parameterization cross.  We write $\PP^1_\Gamma \subset \PP^1$ for the semialgebraic subset parameterizing curves with tropicalization $\Gamma$, and $\cC_\Gamma$ for the preimage of $\PP^1_\Gamma$ in the universal curve $\cC$.  Then $\cC_\Gamma \rightarrow \PP^1_\Gamma$ contains exactly one rational fiber, which is nodal, and it follows that the Euler characteristic of $\cC_\Gamma$ is 1.  The tropicalization of the node in the rational fiber is the tropical node $v$, because the initial degeneration of each fiber at any other point of $\Gamma$ is necessarily smooth.  We will show in \S\ref{sec:genus1} that the motivic volume of $\cC_\Gamma$ is $\LL$ (the class of the affine line in the Grothendieck group of varieties over the residue field), and hence
\[
\chi_{-y}(\cC_\Gamma) = y,
\]
as predicted by Conjecture \ref{conj:rational}. Note that this value is also equal to the $\chi_{-y}$-genus of the rational fiber, so that the $\chi_{-y}$-genus of the union of all smooth fibers is 0.

Next, we consider the tropical curve $\Gamma'$ with a bounded edge of weight 2.  The universal curve $\cC_{\Gamma'} \rightarrow \PP^1_{\Gamma'}$ contains exactly four rational fibers, each of which is nodal, and the tropicalization of each node lies in the interior of the edge of weight 2, because the initial degeneration of each fiber at any other point of $\Gamma'$ is necessarily smooth.  We will show in \S\ref{sec:genus1} that
\[
\chi_{-y}(\cC_{\Gamma'}) = 1 + 2y + y^2,
\]
which is equal to $y N({\Gamma'})$, in agreement with Conjecture~\ref{conj:rational}.  Note that the $\chi_{-y}$-genus of each rational fiber is $y$, so that the $\chi_{-y}$-genus of the union of the smooth fibers is $1 - 2y + y^2$, even though the $\chi_{-y}$-genus of each smooth fiber is 0. This illustrates the crucial complication that, unlike the Euler characteristic, the $\chi_{-y}$-genus of a family cannot be computed by integration on the base, because of the monodromy action on the induced variation of Hodge stuctures. In particular, to compute $\chi_{-y}(\cC_{\Gamma'})$, it is not sufficient to add up the $\chi_{-y}$-genera of the singular fibers.

Similarly, for the tropical curve $\Gamma''$ with a vertex of multiplicity 3, the universal curve $\cC_{\Gamma''} \rightarrow \PP^1_{\Gamma''}$ has exactly three rational fibers, each of which is nodal, and the tropicalization of each node is the vertex of multiplicity 3.  In this case, our computations yield $$\chi_{-y}(\cC_{\Gamma''})=1 + y + y^2,$$ which is equal to $y N({\Gamma''})$, again confirming Conjecture~\ref{conj:rational}.  Once more, we find that the $\chi_{-y}$-genus of the union of all smooth fibers is $1- 2y + y^2$, even though the $\chi_{-y}$ genus of each smooth fiber is 0.
\end{ex}

\noindent \textbf{Acknowledgments.}  We thank F.~Block, L.~G\"ottsche, A.~Gross, I.~Itenberg, D.~Ranganathan, and V.~Shende for helpful conversations related to this work. Johannes Nicaise is supported by the ERC Starting Grant MOTZETA (project 306610) of the European Research Council, and by long term structural funding (Methusalem
grant) of the Flemish Government. Sam Payne is supported in part by NSF CAREER DMS-1149054.

\section{Tropical preliminaries}  \label{sec:prelims}

\subsection{Tropicalizations of curves} \label{sec:tropicalization}

Here we recall a few basic notions from tropical curve counting that are essential for our purposes, including the definition of the Block--G\"ottsche multiplicities.  For further details and references related to tropicalization and analytification of curves, we refer the reader to \cite{BPR16}.

Throughout, we fix a rank two lattice $M$ with dual lattice $N$ and a lattice polygon $\Delta$ in $M_\RR$ with $n +1$ lattice points $u_0, \ldots, u_n$, of which $g$ lie in the interior of $\Delta$. We denote by $Y(\Delta)$ the corresponding toric surface over the field of Puiseux series $K = \Puis$.  Let $f \in K[M]$ be a Laurent polynomial with Newton polygon $\Delta$.  In other words,
\[
f = a_0 x^{u_0} + \cdots + a_n x^{u_n},
\]
with coefficients $a_i \in K$, and $a_i$ is nonzero when $u_i$ is a vertex of $\Delta$.  The vanishing locus of $f$ is a curve $X$ in the torus $T = \Spec K[M]$.  Moreover, since $\Delta$ is the Newton polygon of $f$, the closure $\overline X$ in $Y(\Delta)$ is a curve in the complete linear series $|L(\Delta)|$ that  does not contain any of the $T$-fixed points of $Y(\Delta)$.  Conversely, any curve in $|L(\Delta)|$ that does not contain any of the $T$-fixed points is the closure in $Y(\Delta)$ of the vanishing locus of a Laurent polynomial with Newton polygon $\Delta$.

We associate to $f$ the piecewise linear function $\psi_f : N_\RR \rightarrow \RR$ defined by
\[
\psi_f(v) = \min \{ \langle u_0, v \rangle + \val (a_0), \ldots, \langle u_n, v \rangle + \val (a_n) \}.
\]
The \emph{tropicalization} of the curve $X$, denoted $\Trop(X)$, can be characterized in three equivalent ways, as:
\begin{enumerate}
\item the corner locus of $\psi_f$;
\item the image of the Berkovich analytification $X^\an$, which is a closed subset of $T^\an$, under the proper map $T^\an \rightarrow N_\RR$ defined by coordinatewise valuation;
\item the closure in $N_\RR$ of the image of $X(K)$ under coordinatewise valuation.
\end{enumerate}
It is a rational polyhedral complex of pure dimension $1$.

We also associate to $f$ the \emph{Newton subdivision} of $\Delta$, defined by taking the lower convex hull of the points $(u_i, \val(a_i))$ in $M_\RR \times \RR$, for $a_i \neq 0$, and projecting onto $\Delta$.  Note that the vertices of the Newton subdivision form a subset of the lattice points in $\Delta$, and that this may be a proper subset.  There is a natural order reversing, incidence preserving bijection in which the maximal faces of the Newton subdivision correspond to the vertices of $\Trop(X)$, the edges of the Newton subdivision correspond to the edges of $\Trop(X)$, and the vertices of the Newton subdivision correspond to the chambers of $N_\RR\smallsetminus \Trop(X)$.  This bijection takes a positive-dimensional face $F$ of the Newton subdivision to the face $\gamma_F$ of $\Trop(X)$ given by $$\gamma_F = \{ v \in N_\RR \ | \ \psi_f(v) = \langle u_i, v \rangle + \val (a_i), \mbox{ for } u_i \in F \}.$$  In particular, the unbounded edge directions of $\Trop(X)$ are the inner normals to edges of $\Delta$, and the bounded edge directions are orthogonal to the interior edges of the Newton subdivision.

The edges of $\Trop(X)$ come with positive integer \emph{weights} that satisfy a balancing condition at each vertex, as follows.  If $e$ is an edge of $\Trop(X)$ then the weight $w(e)$ is the lattice length of the corresponding edge of the Newton subdivision.  Let $v$ be a vertex of $\Trop(X)$, let $e_1, \ldots, e_s$ be the edges of $\Trop(X)$ that contain $v$, and let $v_i \in N$ be the primitive lattice vector parallel to $e_i$ in the outgoing direction starting from $v$.  Then the balancing condition at $v$ may be expressed as
\[
w(e_1) v_1 + \cdots + w(e_s) v_s = 0.
\]
Conversely, any rational polyhedral complex of pure dimension 1 in $N_\RR$, with positive integer weights assigned to each edge that satisfy the balancing condition at every vertex, can be realized as the tropicalization of a curve $X \subset T$, and these balanced, weighted complexes are called \emph{tropical curves}.  We say that a tropical curve in $N_\RR$ has \emph{degree $\Delta$} if it is the tropicalization of a curve defined by a Laurent polynomial with Newton polygon $\Delta$.  Equivalently, to each tropical curve, we can associate the multiset of outgoing directions of its unbounded edges, in which the direction of an unbounded edge $e$ is counted $w(e)$ times.  Then a tropical curve has degree $\Delta$ if and only if this multiset consists of the inner normals of $\Delta$, and the number of times that the inner normal of each edge of $\Delta$ appears is the lattice length of that edge.

\subsection{The space of curves with a given tropicalization}\label{sec:fixtrop}

Let $\Gamma$ be a tropical curve of degree $\Delta$. We now describe the locus $|L(\Delta)|_\Gamma$ in the complete linear series $|L(\Delta)|$ parameterizing curves whose intersection with the dense torus $T$ has tropicalization $\Gamma$, as a semialgebraic subset.   The complete linear series $|L(\Delta)|$ is a projective space of dimension $n$, with homogeneous coordinates $[a_0 : \cdots : a_n]$ corresponding to the lattice points $u_0, \ldots, u_n$ in $\Delta$.  The homogeneous coordinate $a_i$ corresponding to a vertex $u_i$ of $\Delta$ must be nonzero for any curve in $|L(\Delta)|$ whose intersection with $T$ has tropicalization $\Gamma$.  In particular, we may assume that $u_0$ is a vertex of $\Delta$ and restrict attention to the affine space where $a_0$ is nonzero, and normalize so that $a_0 = 1$.  Then $a_1, \ldots, a_n$ are coordinates on this affine space.  Note that there is a unique concave function $\phi: \Delta \cap \ZZ^2 \to \QQ$ such that $\phi(u_0) = 0$ and the corner locus of the corresponding concave piecewise linear function $\psi$ on $\RR^2$ given by
\[
\psi(v) = \min_{u \in \Delta \cap \ZZ^2} \langle u, v \rangle + \phi(u)
\]
is exactly $\Gamma$, with its given edge weights.  Then, the tropicalization of the curve $X$ in $T$ defined by the equation $x^{u_0} + a_1 x^{u_1} + \cdots + a_n x^{u_n}= 0$ is equal to $\Gamma$ if and only if $\val(a_i) = \phi(u_i)$ for each vertex $u_i$ of the Newton subdivision dual to $\Gamma$, and $\val(a_j) \geq \phi(u_j)$ for each lattice point $u_j$ in $\Delta$ that is not a vertex of the Newton subdivision. These conditions define a closed semialgebraic subset of the complete linear series $|L(\Delta)|$.  See \S\ref{sec:semialgebraicsets} for further details on semialgebraic sets and \cite{KP11, Kat13} for a more general treatment of \emph{realization spaces} for tropical varieties.

\subsection{Tropical curve counting}  \label{sec:tropicalcounting}

Tropical curves of degree $\Delta$ in $N_\RR$ can be used to count algebraic curves in $|L(\Delta)|$, using the well-known correspondence theorems.  To explain the statement of these theorems, it is helpful first to define the \emph{genus} of a tropical curve using parameterizations.

Let $\Gamma$ be a tropical curve of degree $\Delta$ in $N_\RR$.  A parameterization of $\Gamma$ is a metric graph $\Gamma'$ with a continuous surjective map $p: \Gamma' \rightarrow \Gamma$ such that the restriction of $p$ to each edge of $\Gamma'$ is linear with derivative in the lattice $N$.  These slopes are required to satisfy the balancing condition at each vertex of $\Gamma'$, meaning that the sum of the outgoing slopes at each vertex is zero.  Finally, the parameterization must be compatible with the edge weights on $\Gamma$.  This means that there is some subdivision of $\Gamma$ in which each edge is the homeomorphic image of finitely many edges of $\Gamma'$, and each of these homeomorphisms is a dilation by a positive integer factor (with respect to the lattice metric on $\Gamma$), such that the sum of these dilation factors is the edge weight in $\Gamma$.  Note that parameterizations may contract some edges of $\Gamma'$. We identify two parameterizations if a subdivision of one is isomorphic to a subdivision of the other.  The \emph{genus} of a tropical curve $\Gamma$ in $N_\RR$ is defined to be the smallest first Betti number of a metric graph $\Gamma'$ that admits a parameterization $p: \Gamma' \rightarrow \Gamma$.  It follows easily from the theory of skeletons of Berkovich analytifications of curves, as discussed in \cite{BPR16}, that the genus of the tropicalization of a curve $X \subset T$ is less than or equal to the geometric genus of $X$.

Fix $0 \leq \delta \leq g$, and choose a set $S$ of $n - \delta$ rational points in general position in $N_\RR$. Note that we do not require $S$ to be in vertically or horizontally stretched position, or to lie on a line of irrational slope, we just require that it be general with respect to the appropriate evaluation map, as in \cite{GM08}.  This condition is satisfied on an open dense subset of $N_\RR^{n- \delta}$.  One of the essential and foundational combinatorial facts underlying tropical curve counting is that there are only finitely many parameterized tropical curves of genus $g - \delta$ and degree $\Delta$ passing through $S$.  Each of these tropical curves is \emph{simple}, meaning that the parameterizing curve is connected and trivalent, the parameterization is an immersion, the image has only trivalent and 4-valent vertices, and the preimage of each 4-valent vertex has exactly two points.  Moreover, each unbounded edge has weight 1.

Now, choose a collection $\widetilde S$ of $n - \delta$ points in $T(K)$ whose tropicalization is $S$. There are finitely many curves of geometric genus $g - \delta$ in $|L(\Delta)|$ that contain $\widetilde S$, and the tropicalization of each of these is one of the finitely many tropical curves of genus $g - \delta$ that contains $S$.  Let $\Gamma$ be one of these tropical curves.  The nonarchimedean tropical correspondence theorems tell us that the number of algebraic curves of genus $g - \delta$ that contain $\widetilde S$ and tropicalize to $\Gamma$ can be expressed combinatorially as a product over trivalent vertices,
\[
n(\Gamma) = \prod_v m(v)
\]
where the factor $m(v)$ corresponding to a vertex $v$ is the index of the sublattice of $N$ generated by the vectors $w(e_1) v_1, \ldots, w(e_s) v_s$, where $v_i$ is the primitive lattice vector in the outgoing direction along the edge $e_i$, and $w(e_i)$ is the weight of this edge.

The Block--G\"ottsche multiplicities are defined similarly, by the combinatorial formula
\[
N(\Gamma) = \prod_v M(v),
\]
where $M(v)$ is the Laurent polynomial in $y^{1/2}$ given by $$\frac{y^{m(v)/2}-y^{-m(v)/2}}{y^{1/2}-y^{-1/2}} = y^{(m(v)-1)/2}+y^{(m(v)-3)/2}+\cdots+y^{-(m(v)-1)/2}.$$  Each tropical curve of genus $g - \delta$ containing $S$ has an even number of vertices such that $m(v)$ is even.  To see this, note that the multiplicity of a trivalent vertex in the embedded tropical curve is twice the area of the corresponding triangle in the Newton subdivision.  Pick's formula then implies that the multiplicity of the vertex is congruent to the lattice perimeter of the corresponding triangle, modulo 2.  Therefore, if every edge of Gamma has odd weight, then every triangle has odd perimeter.  Each bounded edge of even weight in the tropical curve corresponds to an interior edge of even length in the Newton subdivision, which changes the parity of the lattice perimeter of two triangles.  It follows that $N(\Gamma)$ is a Laurent polynomial in $y$, even though it is expressed as a product of Laurent polynomials in $y^{1/2}$. This Laurent polynomial specializes to $n(\Gamma)$ by setting $y = 1$, since $M(v)$ specializes to $m(v)$ by setting $y^{1/2} = 1$.

\section{Motivic volumes of semialgebraic sets}  \label{sec:motivic}
 An essential tool in this paper is a motivic Euler characteristic for definable sets in the language of valued fields over $\Puis$, which was constructed by Hrushovski-Kazhdan \cite{HK06} using deep results from model theory. It assigns a value in the Grothendieck ring $K_0(\Var_{\CC})$ of  complex varieties to any semialgebraic set in an algebraic $\Puis$-variety, and it is additive with respect to disjoint unions. We will use this motivic Euler characteristic to define the $\chi_{-y}$-genus of semialgebraic sets such as $\Jac(\cC_\Gamma)$.

 Hrushovski and Kazhdan's theory works over any henselian valued field of equal characteristic zero. We will review the main statements in the special case we need, where the field is real-valued and algebraically closed.  The latter assumption  substantially simplifies a part of the construction, by collapsing the generalized residue field structure. Although the proofs of these results use the model theory of algebraically closed valued fields in an essential way, we have tried to present the statements, as much as possible, in a geometric language.

 Unless explicitly stated otherwise, the results in this section are valid for any height one valuation ring $R$ of equal characteristic zero with algebraically closed quotient field $K$. We denote by $k$ the residue field of $R$, and by $G$ the value group of $K$.  The assumption that $K$ is algebraically closed implies that $k$ is algebraically closed and $G$ is divisible. The valuation map is denoted by $\val:K^\ast \twoheadrightarrow G$.  We fix an embedding of the ordered group $G$ in $(\RR,+,\leq)$, extend the ordering to $G\cup \{\infty\}$ by declaring that $x\leq \infty$ for every $x$ in $G\cup \{\infty\}$, and extend the valuation to $K$ by setting $\val(0) = \infty$.

  A polyhedron in $G^n$ is an intersection of finitely many half-spaces of the form $$\{u\in G^n\,|\,a_1u_1+\ldots +a_nu_n\leq a_0\}$$ with $a_1,\ldots,a_n$ in $\QQ$ and $a_0$ in $G$. We will also use the term ``$G$-rational polyhedron'' to denote subsets of $\RR^n$ defined by the same type of formulas. This should not lead to confusion, since a $G$-rational polyhedron in $\RR^n$ is completely determined by its intersection with the dense subset $G^n$ of $\RR^n$. For every integer $n>0$, we denote by $\trop$ the tropicalization map
  $$\trop:(K^{\ast})^n\to G^n:(x_1,\ldots,x_n)\mapsto (v(x_1),\ldots,v(x_n)).$$ More generally, for any algebraic torus $T$ over $K$ with cocharacter lattice $N$, we denote by
  $$\trop:T(K)\to N_{\RR}$$ the map defined by coordinatewise valuation.

\subsection{Semialgebraic subsets of algebraic $K$-varieties} \label{sec:semialgebraicsets}
Let $n$ be a positive integer. A semialgebraic subset of $K^n$ is a finite Boolean combination of subsets of the form $$\{x\in K^n\,|\,\val(f(x))\geq \val(g(x))\}$$ where $f$ and $g$ are polynomials in $K[x_1,\ldots,x_n]$. More generally, if $X$ is a $K$-scheme of finite type, then a subset $S$ of $X(K)$ is called a semialgebraic subset of $X$ if we can cover $X$ by affine open subschemes $U$ such that there exists a closed immersion $i:U\to \AA^n_K$, for some $n>0$, with the property that $i(S\cap U(K))$ is a semialgebraic subset of $K^n$. It is easy to check that, if $S$ is a semialgebraic subset of $X$, then $i(S\cap U(K))$ is a semialgebraic subset of $K^n$ for {\em every} open subscheme $U$ of $X$ and every closed immersion $i:U\to \AA^n_K$. Robinson's  quantifier elimination theorem for algebraically closed valued fields \cite{Robinson56} implies that, if $f:X\to Y$ is a morphism of $K$-schemes of finite type and $S$ is a semialgebraic subset of $X$, then $f(S)\subset Y(K)$ is a semialgebraic subset of $Y$.
 \begin{examples}\label{exam:semialg}
\begin{enumerate}
\item If $X$ is a $K$-scheme of finite type, then every constructible subset of $X(K)$ is semialgebraic. Indeed, locally on $X$, it is a Boolean combination of subsets of the form
    $$\{x\in X(K)\,|\,f(x)=0\}=\{x\in X(K)\,|\,\val(f(x))\geq \val(0)\}$$ with $f$ a regular function.

\item \label{it:spsemialg} Let $\cX$ be an $R$-scheme of finite type and set $X=\cX_K$. The specialization map
    $$\spe_{\cX}:\cX(R)\to \cX(k)$$ is defined by reducing coordinates modulo the maximal ideal of $R$.
If $C$ is a constructible subset of $\cX(k)$, then $\spe_{\cX}^{-1}(C)$ is a semialgebraic subset of $X$. To prove this, it suffices to consider the case where $\cX$ is affine and $C$ is closed in $\cX(k)$. If $t_1,\ldots,t_n$ generate the $R$-algebra $\mathcal{O}(\cX)$ and $C$ is the set of closed points of the zero locus of an ideal $(f_1,\ldots,f_\ell)$ in $\mathcal{O}(\cX)$, then
 $$\spe_{\cX}^{-1}(C)=\{x\in X(K)\,|\,\val(t_i(x))\geq 0 \mbox{ and }\val(f_j(x))>0\mbox{ for all }i,j\}.$$
 When $C$ is a constructible subset of $\cX_k$ (rather than $\cX(k)$) we will write $\spe_{\cX}^{-1}(C)$ instead of $\spe_{\cX}^{-1}(C\cap \cX(k))$.

\item If $a_1,\ldots,a_n$ are elements of $\Q$ and $c$ is an element of $G$, then $$\{x\in (K^\ast)^n\,|\,a_1\val(x_1)+\ldots+a_n\val(x_n)\leq c\}$$ is a semialgebraic subset of $K^n$. More generally, if $S$ is a finite Boolean combination of $G$-rational polyhedra in $\RR^n$, then $\trop^{-1}(S)$ is a semialgebraic subset of $K^n$. Conversely, it follows from Robinson's quantifier elimination for algebraically closed valued fields that the image of every semialgebraic subset of $(K^\ast)^n$ under $\trop$ is a finite Boolean combination of $G$-rational polyhedra in $G^n$.

\end{enumerate}
  \end{examples}

We will now discuss some less obvious examples that are of interest in tropical geometry.

\begin{prop}
Let $Y$ be a $K$-scheme of finite type and let $X$ be a subscheme of $Y\times_K \mathbb{G}_{m,K}^n$, for some $n>0$. We denote by $f:X(K)\to Y(K)$ the restriction of the projection morphism $Y\times_K \mathbb{G}_{m,K}^n\to Y$. Let $\Gamma$ be a finite Boolean combination of $G$-rational polyhedra in $\RR^n$. Then the set of points $y$ in $Y(K)$ such that $\trop(f^{-1}(y))=\Gamma$ is a semialgebraic subset of $Y$.
\end{prop}

\begin{proof}
This is an immediate consequence of Robinson's  quantifier elimination theorem for algebraically closed valued fields.
\end{proof}

\begin{prop}
Let $Y$ be a $K$-scheme of finite type and let $X$ be a hypersurface in $Y\times_K \mathbb{G}_{m,K}^n$, for some $n>0$. We denote by $f:X(K)\to Y(K)$ the restriction of the projection morphism $Y\times_K \mathbb{G}_{m,K}^n\to Y$.
 Let $T$ be any topological space. Then the set of points $y$ in $Y(K)$ such that the closure of $\trop(f^{-1}(y))$ in $\RR^n$ is homeomorphic to $T$ is a semialgebraic subset of $Y$.
\end{prop}

\begin{proof}
Given a family of hypersurfaces in the torus $\mathbb{G}_{m,K}^n$, after stratifying the base by the Newton polytope of the defining equation, we may assume that the Newton polytope is constant.

Now, the tropicalization of a hypersurface in $\mathbb{G}_{m,K}^n$ has the structure of a regular polyhedral complex, which is dual to the Newton subdivision of the Newton polytope. In particular, the homeomorphism type of the tropical hypersurface is determined by the Newton subdivision.  There are finitely many possibilities for the Newton subdivision, and the condition of the defining equation having fixed Newton subdivision is semialgebraic, given simply by linear inequalities on the valuations of the coefficients of the defining equation.
\end{proof}

\begin{rem}
 Although we will not need it in this paper, we want to point out a related result involving Berkovich skeletons of stable marked curves over $K$. Let $f:X\to Y$ be a morphism of $K$-schemes of finite type, and let $\sigma_1,\ldots,\sigma_n:Y\to X$ be sections of $f$ whose images are disjoint. Assume that $f^{-1}(y)$, marked with the points $(\sigma_1(y),\ldots,\sigma_n(y))$ is a stable curve over $K$, for every $y$ in $Y(K)$. Then for every graph $\Gamma$, the set of points $y$ in $Y(K)$ such that the skeleton of $f^{-1}(y)$ is homeomorphic to $\Gamma$ is a semialgebraic subset of $Y$. This can be deduced from Example \ref{exam:semialg}\eqref{it:spsemialg} and the fact that the combinatorial types of Berkovich skeletons of stable genus $g$ curves with $n$ marked points correspond to the boundary strata of $\overline{\mathcal{M}}_{g,n}$ \cite{acp}.
\end{rem}

\subsection{The Grothendieck ring of semialgebraic sets}
 We define the category $\VF_K$ of semialgebraic sets over $K$ as follows. The objects in this category are the pairs $(X,S)$ with $X$ a $K$-scheme of finite type and $S$ a semialgebraic subset of $X$. A morphism $(X,S)\to (Y,T)$ in this category is a map $S\to T$ whose graph is a semialgebraic subset of $X\times_K Y$. It is clear that a composition of two such maps is again a morphism in $\VF_K$. We will often denote an object $(X,S)$ in the category $\VF_K$ simply by $S$, leaving $X$ implicit. The image and inverse image of a semialgebraic set under a morphism in $\VF_K$ are again semialgebraic.

 The Grothendieck group $K_0(\VF_K)$ of semialgebraic sets over $K$ is the free abelian group on isomorphism classes $[S]$ of semialgebraic sets $S$ over $K$ modulo the relations $$[S]=[T]+[S\smallsetminus T]$$ for all $K$-schemes of finite type $X$ and all semialgebraic subsets $T\subset S$ of $X$. Note that, in particular, $[\emptyset]=0$. The group $K_0(\VF_K)$ has a unique ring structure such that $$[S]\cdot [S']=[S\times S']$$ in $K_0(\VF_K)$ for all semialgebraic sets $S$ and $S'$.

\subsection{The Grothendieck ring of algebraic varieties}
For every field $F$, the Grothendieck ring of $F$-varieties $K_0(\Var_F)$ is the free abelian group on isomorphism classes $[X]$ of $F$-schemes $X$ of finite type modulo the scissor relations $[X]=[Y]+[X\smallsetminus Y]$ for all $F$-schemes of finite type $X$ and all closed subschemes $Y$ of $X$. The ring structure on $K_0(\Var_F)$ is determined by the property that $$[X]\cdot [X']=[X\times_F X']$$ for all $F$-schemes of finite type $X$ and $X'$. We set $\LL=[\AA^1_F]$. If $F$ is algebraically closed and of characteristic zero, which is the only case we will need, then $K_0(\Var_F)$ is canonically isomorphic to the Grothendieck ring of definable sets over the field $F$ (see \S3.7 and \S3.8 in \cite{NiSe-K0}). It is clear from the definitions that there exists a unique ring morphism $$K_0(\Var_K)\to K_0(\VF_K)$$ that sends the class of a $K$-scheme of finite type $X$ in $K_0(\Var_K)$ to the class of $X(K)$ in  $K_0(\VF_K)$ (where we view $X(K)$ as a semialgebraic subset of $X$).

For every integer $n\geq 0$, we define the Grothendieck group $K_0(\Var_F[n])$ in the same way as $K_0(\Var_F)$, except that we only consider $F$-schemes of finite type $X$, $Y$ of dimension at most $n$ in the definition of the generators and the scissor relations.

For every $F$-scheme of finite type $X$ of dimension at most $n$, we will denote by $[X]_n$ its class in $K_0(\Var_F[n])$. We give the direct sum $$K_0(\Var_F[\ast])=\bigoplus_{n\geq 0}K_0(\Var_F[n])$$ the structure of a graded ring by setting $$[X]_m\cdot [X']_n=[X\times_F X']_{m+n}$$ for all integers $m,n\geq 0$ and all $F$-schemes of finite type $X,X'$ of dimension at most $m$, respectively $n$.

\subsection{The Grothendieck ring of $G$-rational polyhedra}
For every nonnegative integer $n$, we write $G[n]$ for the category of definable subsets in $G^n$ with integral affine transformations. A definable subset of $G^n$ is a finite Boolean combination of $G$-rational polyhedra in $G^n$. A morphism $S\to T$ in $G[n]$ is a bijective map $f:S\to T$ such that there exists a finite partition of $S$ into definable subsets $S_i$ satisfying the property that the restriction of $f$ to $S_i$ can be written as $x\mapsto Mx+b$ with $M\in \mathrm{GL}_n(\ZZ)$ and $b\in G^n$. In particular, every morphism in this category is an isomorphism.
   Now the Grothendieck group $K_0(G[n])$ is the free abelian group on isomorphism classes $[S]_n$ of definable subsets $S$ of $G^n$ modulo the relations $$[S]_n=[T]_n+[S\smallsetminus T]_n$$ for all definable subsets $T\subset S$ of $G^n$. We again give the direct sum $$K_0(G[\ast])=\bigoplus_{n\geq 0}K_0(G[n])$$ the structure of a graded ring by setting $$[S]_m\cdot [T]_n=[S\times T]_{m+n}$$ for all definable subsets $S,T$ of $G^m$ and $G^n$, respectively. The class of a point in $G^n$ will be denoted by $[1]_n$.

Every definable subset of $G^n$ has a natural dimension by the theory of $o$-minimal structures. It is equal to the smallest non-negative integer $d$ such that $S$ lies in a finite union of spaces of the form
  $$(V\otimes_{\QQ}\RR + g)\cap G^n$$ where $V$ is a $d$-dimensional linear subspace of $\QQ^n$ and $g$ is an element of $G^n$.    This notion of dimension agrees with the topological dimension of the closure of $S$ in $\RR^n$. By convention, the empty set has  dimension $-\infty$. If $S$ is a definable subset of $G^n$ of dimension at most $m$, it is not difficult to see that there exists a finite partition of $S$ into definable subsets $S_1,\ldots,S_r$ such that, for each $i$, there exists an isomorphism in $G[n]$ between $S_i$ and a definable subset $T_i$ of $G^m\subset G^n$. The element $[T_1]+\cdots +[T_r]$ of $K_0(G[m])$ does not depend on the choice of the partition or the sets $T_i$, and will be denoted by $[S]_m$.

 Beware that a homothety with factor different from $\pm 1$ is {\em not} a morphism in $G[n]$ for $n>0$. Nevertheless, we can still make the following identifications in the Grothendieck group.

\begin{prop}\label{prop:scale}
For every positive integer $n$ and every element $g>0$ in $G$, we denote by $\Delta^o_{n,g}$ the open simplex
 $$\Delta^o_{n,g}=\{x\in G^{n}\,|\,\sum_{i=1}^n x_i<g\mbox{ and }x_i>0\mbox{ for all }i\}.$$
Then $$[\Delta^o_{n,g}]_n=(-1)^{n}[1]_n$$ in $K_0(G[n])$.
\end{prop}

\begin{proof}
We will proceed by induction on $n$. For $n=1$, the intervals $[0,\infty)$ and $[g,\infty)$ in $G$ are isomorphic in $G[1]$ (via the translation over $g$) so that the class of their difference $[0,g)$ vanishes in $K_0(G[1])$. Hence, the class in $K_0(G[1])$ of every semi-open bounded interval with endpoints in $G$ is equal to zero. Thus the class of any bounded open interval with endpoints in $G$ is equal to minus $[1]_1$, the class of a point.

Now assume that $n>1$ and that the result holds for all strictly smaller values of $n$. We consider the definable subsets
   \begin{eqnarray*}S_0&=&\{x\in G^{n}\,|\,\sum_{i=1}^{n-1} x_i<g\mbox{ and }x_i>0\mbox{ for }i=1,\ldots,n\}
   \\S_1&=&\{x\in G^{n}\,|\,\sum_{i=1}^{n} x_i=g\mbox{ and }x_i>0\mbox{ for }i=1,\ldots,n\}
   \\S_2&=&\{x\in S_0\,|\,\sum_{i=1}^{n} x_i>g \}
   \end{eqnarray*}
of $G^n$. Then $\{\Delta^o_{n,g},S_1,S_2\}$ is a partition of $S_0$, and $S_0$ and $S_2$ are isomorphic in $G[n]$. Thus,
   $$[\Delta^o_{n,g}]_n=-[S_1]_n$$ in $K_0(G[n])$. But $S_1$ is isomorphic to $\Delta^o_{n-1,g}\times \{0\}$ in $G[n]$, so that
 $$[S_1]_n=[1]_1\cdot [\Delta^o_{n-1,g}]_{n-1}=(-1)^{n-1}[1]_n$$ by the induction hypothesis. This concludes the proof.
 \end{proof}

\subsection{The motivic volume of Hrushovski-Kazhdan}\label{sec:KVF}

We will now recall how the theory of motivic integration of Hrushovski-Kazhdan \cite{HK06} gives rise to a ring morphism
$$\Vol:K_0(\VF_K)\to K_0(\Var_k).$$ If $S$ is a semialgebraic set over $K$, then we will write $\Vol(S)$ for $\Vol([S])$ and we call this object the {\em motivic volume} of $S$.
 If $X$ is a $K$-scheme of finite type, then we write $\Vol(X)$ instead of $\Vol(X(K))$.
 In order to keep our presentation accessible to readers with a background in algebraic geometry, we will state the construction in slightly different terms than those used in \cite{HK06}, but it is not difficult to see that it yields the same result.

One of the central results of Hrushovski and Kazhdan's theory is the construction of a ring isomorphism
\begin{equation}\label{eq:theta}
\Theta':K_0(\VF_K)\to (K_0(\RES_K[\ast])\otimes_{\ZZ[\tau]}K_0(G[\ast]))/I_{\spe}.
\end{equation}
Let us explain the different notations, as well as the main ingredients of this construction. Since, in our setting, the value group of $K$ is divisible and $k$ has characteristic zero, the Grothendieck ring $K_0(\RES_K[\ast])$ is simply the graded ring
$$K_0(\Var_k[\ast])=\bigoplus_{n\geq 0}K_0(\Var_k[n]).$$
We view $K_0(\RES_K[\ast])$ and $K_0(G[\ast])$ as graded $\ZZ[\tau]$-algebras {\em via} the morphisms
$$\begin{array}{l}
\ZZ[\tau]\to K_0(G[\ast]):\tau\mapsto [1]_1,
\\[2pt] \ZZ[\tau]\to  K_0(\RES_K[\ast]): \tau\mapsto [\mathbb{G}_{m,k}]_1.
\end{array}$$
Hrushovski and Kazhdan defined a ring morphism $$\Theta:K_0(\RES_K[\ast])\otimes_{\ZZ[\tau]}K_0(G[\ast]))\to K_0(\VF_K)$$ that is characterized by the following properties. If $S$ is a definable subset of $G^n$, then  the morphism $\Theta$ maps  $[S]_n\in K_0(G[n])$ to the class in $K_0(\VF_K)$ of the semialgebraic subset $\trop^{-1}(S)$ of $K^n$. By Noether normalization, we can write every element in $K_0(\RES_K[n])$ as a $\ZZ$-linear combination of classes of $k$-schemes of finite type $X$ that admit a quasi-finite morphism $f:X\to \AA^n_k$ with locally closed image. Since $k$ has characteristic zero, we can assume that $X$ is \'etale over $f(X)$ (endowed with its reduced induced structure).
Then by \cite[18.1.1]{ega4.4}, we can even suppose that there exists an \'etale $\AA^n_R$-scheme of finite type $\cX$ and a closed immersion $X\to \cX_k$ such that the restriction of $\cX_k\to \AA^n_k$ to $X$ coincides with $f$. The morphism $\Theta$ maps $[X]_n \in K_0(\Var_k[n])$ to the class in $K_0(\VF_K)$ of the semialgebraic set $\spe_{\cX}^{-1}(X)$. This definition is independent of the choice of $\cX$, because $R$ is henselian.

Hrushovski and Kazhdan proved in \cite[8.8 and 10.3]{HK06} that the morphism $\Theta$ is surjective, and that its kernel is the ideal $I_{\spe}$ generated by the element $$[\Spec k]_0+[G_{>0}]_1-[\Spec k]_1.$$  Note that this element indeed belongs to the kernel of $\Theta$ since the image of $[\Spec k]_0$ is the class of a point in $K_0(\VF_K)$, the image of $[G_{>0}]_1$ is the class of the punctured open unit disc $$\{x\in K^\times\,|\,\val(x)>0\},$$ and the image of $[\Spec k]_1$ is the class of the open unit disc $\spe^{-1}_{\AA^1_R}(0)$ in $K$. Thus the morphism $\Theta$ factors through an isomorphism $$(K_0(\RES_K[\ast])\otimes_{\ZZ[\tau]}K_0(G[\ast]))/I_{\spe}\to  K_0(\VF_K),$$ and $\Theta'$ is, by definition, the inverse of this isomorphism.

The morphism $$\Vol:K_0(\VF_K)\to K_0(\Var_k)$$ is now obtained by composing $\Theta'$ with the ring morphism
 $$\Psi:(K_0(\RES_K[\ast])\otimes_{\ZZ[\tau]}K_0(G[\ast]))/I_{\spe}\to K_0(\Var_k)$$
defined in \cite[10.5(4)]{HK06} (we implicitly make use of the isomorphism in \cite[10.3]{HK06}). Let us unravel its construction. For every $n\geq 0$ we have an obvious group morphism $$K_0(\Var_k[n])\to K_0(\Var_k):[X]_n\mapsto [X]$$ and these give rise to a ring morphism
\begin{equation}\label{eq:res}
K_0(\RES_K[\ast])\to K_0(\Var_k).
\end{equation}
On the other hand, we can consider the Euler characteristic
  $$\chi':K_0(G[\ast])\to \ZZ$$
from \cite[9.6]{HK06} that sends the class in $K_0(G[n])$ of a definable subset $S$ of $G^n$ to
  $$\chi'(S)=\lim_{r\to +\infty}\chi(S\cap [-r,r]^{n}),$$
where $\chi(\cdot)$ is the usual $o$-minimal Euler characteristic \cite[Ch4\S2]{vandenDries}.
The value of $\chi(S\cap [-r,r]^{n})$ stabilizes for sufficiently large $r$.

Equivalently, the ring morphism $\chi'$ is characterized by the property that it sends the class of any closed polyhedron to $1$; the image of the class of an arbitrary definable set can then be computed by cell decomposition, using the additivity of $\chi'$.

\begin{prop}\label{prop:chi-poly}
Let $\gamma$ be a non-empty $G$-rational polyhedron in $\RR^n$, for some $n\geq 0$, and denote by $\mathring{\gamma}$ its relative interior.
Then $$\chi'(\mathring{\gamma}\cap G^n)=(-1)^{\mathrm{dim}(\gamma)}$$ if $\gamma$ is bounded;  $\chi'(\mathring{\gamma}\cap G^n)=1$ if $\gamma$ is an affine subspace of $\RR^n$; and $\chi'(\mathring{\gamma}\cap G^n)=0$ in all other cases.
\end{prop}
\begin{proof}
 In the bounded case, $\chi'(\mathring{\gamma}\cap G^n)$ coincides with the usual Euler characteristic  (with compact supports)
 of $\mathring{\gamma}$, and the assertion is obvious.

  In the unbounded case, $$\chi'(\mathring{\gamma}\cap G^n)=\chi(\mathring{\gamma} \cap [-r,r]^{\mathrm{dim}(\gamma)})$$ when $r$ is sufficiently large.
  If $\gamma$ is an affine subspace of $\RR^n$, then $\mathring{\gamma}=\gamma$ and $\gamma \cap [-r,r]^{\mathrm{dim}(\gamma)}$ is homeomorphic to $[0,1]^{\mathrm{dim}(\gamma)}$, so that the Euler characteristic equals $1$. Otherwise,
   for large $r$, the space $\mathring{\gamma}\cap\, (-r,r)^{\mathrm{dim}(\gamma)}$ is homeomorphic to $\RR^{\mathrm{dim}(\gamma)}$ while
   the relative boundary of $\mathring{\gamma}\cap [-r,r]^{\mathrm{dim}(\gamma)}$ is homeomorphic to $\RR^{\mathrm{dim}(\gamma)-1}$, so that $$\chi(\mathring{\gamma} \cap [-r,r]^{\mathrm{dim}(\gamma)})=0.$$
 \end{proof}

We use $\chi'$ to define a group morphism \begin{equation}\label{eq:chi}
   K_0(G[n])\to K_0(\Var_k):\alpha\mapsto \chi'(\alpha)(\LL-1)^n,
\end{equation}
for every $n\geq 0$.  The morphisms \eqref{eq:res} and \eqref{eq:chi} together induce a ring morphism
   $$K_0(\RES_K[\ast])\otimes_{\ZZ[\tau]}K_0(G[\ast])\to K_0(\Var_k),$$
whose kernel contains $I_{\spe}$ because $\chi'(G_{>0})=0$. Thus it factors through a ring morphism
   $$\Psi:(K_0(\RES_K[\ast])\otimes_{\ZZ[\tau]}K_0(G[\ast]))/I_{\spe}\to K_0(\Var_k).$$

It will be useful to consider the behavior of the motivic volume under extensions of the valued field $K$.  Let $K'$ be an algebraically closed valued extension of $K$ of rank 1. The case of most interest to us will be the case where $K'$ is the completion of $K$.  For every semialgebraic set $S$ over $K$, we can consider its base change $S\times_K K'$ to $K'$, which is defined by the same formulas (this construction is well-defined because of quantifier elimination).

\begin{prop}\label{prop:complete}
 For every semialgebraic set $S$ over $K$, we have
 $\Vol(S\times_K K')=\Vol(S)$.
\end{prop}
\begin{proof}
 This follows easily from the definition of the motivic volume, since the morphism $\Theta$ is compatible with base change from $K$ to $K'$.
\end{proof}

 \subsection{Tropical computation of the motivic volume}\label{sec:tropcomp}
In this section, we will prove some properties that can be used to compute the motivic volume in concrete cases.
 The most basic example is the following.
\begin{prop}\label{prop:smoothvol}
If $\cX$ is a smooth $R$-scheme of finite type of pure relative dimension $d$ and $Y$ is a constructible subset of $\cX_k$, then
$$[\spe_{\cX}^{-1}(Y)]=\Theta([Y]_d)$$ in $K_0(\VF_K)$, and
$$\Vol(\spe_{\cX}^{-1}(Y))=[Y]$$ in $K_0(\Var_k)$.

In particular, $[\cX(R)]=\Theta([\cX_k]_d)$ in $K_0(\VF_K)$, and $\Vol(\cX(R))=[\cX_k]$ in $K_0(\Var_k)$.
\end{prop}

\begin{proof}
By additivity, we may assume that $Y$ is closed in $\cX_k$ and that $\cX$ admits an \'etale morphism to $\AA^d_R$ for some $d \geq 0$. Then, by construction, the morphism $\Theta$ maps the class of $Y$ in $K_0(\Var_k[d])$ to the class of $\spe_{\cX}^{-1}(Y)$ in $K_0(\VF_K)$, and it follows from the definition of the morphism $\Vol$ that $\Vol(\spe_{\cX}^{-1}(Y))=[Y]$.
\end{proof}

 The definition of the motivic volume also makes it well adapted to tropical computations. Let $X$ be a reduced  closed subscheme of dimension $d$ of a torus $T\cong \mathbb{G}^n_{m,K}$, for some $n>0$. We denote by $M$ and $N$ the character lattice and cocharacter lattice of $T$, respectively, and by $\mathrm{Trop}(X)\subset N_{\RR}$ the tropicalization of $X$. We denote by $\TT$ the split $R$-torus with character lattice $M$. Let $\Sigma$ be a {\em $G$-admissible  tropical fan} for $X$ in $N_\RR\oplus \RR_{\geq 0}$ in the sense of \cite[12.1]{G13} (henceforth, we will simply speak of a {\em tropical fan}). It defines a toric scheme  $\PP(\Sigma)$ over $R$.     If we write $\cX$ for the schematic closure of $X$ in $\PP(\Sigma)$, then $\cX$ is proper over $R$ and the multiplication map $$m:\TT\times_R \cX\to \PP(\Sigma)$$ is faithfully flat. Recall that $X$ is called {\em sch\"{o}n} if the initial degeneration $\mathrm{in}_w(X)$ is smooth over $k$ for every $w\in N_{\RR}$; this is equivalent to saying that $m$ is smooth.

Intersecting the cones of $\Sigma$ with $N_\RR\times\{1\}$, we obtain a $G$-rational polyhedral complex in $N_\RR$ that we denote by $\Sigma_1$. The support of $\Sigma_1$ is equal to $\Trop(X)$, by \cite[12.5]{G13}.  For every cell $\gamma$ in $\Sigma_1$, we denote its relative interior by $\mathring{\gamma}$.  By \cite[12.9]{G13}, all the points $w$ in $\mathring{\gamma}$ give rise to the same initial degeneration $\mathrm{in}_w X$, which we will denote by $\mathrm{in}_{\gamma}X$. We write $X_{\gamma}$ for the semialgebraic subset
   $$X(K)\cap \trop^{-1}(\mathring{\gamma})$$
of $X$. As $\gamma$ ranges over the cells in $\Sigma_1$, the sets $X_{\gamma}$ form a partition of $X(K)$. We denote by $\cX_k(\gamma)$ the intersection of $\cX_k$ with the torus orbit of $\PP(\Sigma)_k$ corresponding to the cell $\gamma$. Then we can also write $X_\gamma$ as
$$X_\gamma=\spe^{-1}_{\cX}(\cX_k(\gamma))\cap X(K).$$

 \begin{prop}\label{prop:tropvol}
 Let $X$ be a reduced  closed subscheme of dimension $d$ of a $K$-torus $T$ and let $\Sigma$ be a tropical fan for $X$. Then, with the above notations, we have
 $$[X_{\gamma}]=\Theta([\cX_k(\gamma)]_{d-\mathrm{dim}(\gamma)}\otimes [\mathring{\gamma}]_{\mathrm{dim}(\gamma)})\quad
\in K_0(\VF_K)$$ for every cell $\gamma$ in the polyhedral complex $\Sigma_1$ such that $\mathrm{in}_{\gamma}X$ is smooth over $k$. In particular, if $X$ is sch{\"o}n, then
$$[X(K)]=\sum_{\gamma\in \Sigma_1}\Theta([\cX_k(\gamma)]_{d-\mathrm{dim}(\gamma)}\otimes [\mathring{\gamma}]_{\mathrm{dim}(\gamma)}).$$
 \end{prop}

\begin{proof}
 Let $\gamma$ be a cell in $\Sigma_1$.
 Translating $X$ by a point of $T(K)$, we can assume that $\mathring{\gamma}$ contains the origin of $N_\RR$. Let $V$ be the linear subspace of $N_\RR$ generated by $\gamma$ and let $\widetilde{\TT}$ be the sub-$R$-torus of $\TT$ with cocharacter lattice $V\cap N$. Its generic fiber will be denoted by $\widetilde{T}$, and we write $\widetilde{T}_\gamma$ for the inverse image of $\mathring{\gamma}$ under the tropicalization map
 $$\trop:\widetilde{T}(K)\to V.$$

\smallskip
{\em Step 1.} Let $g_K$ be any point of $\widetilde{T}_\gamma$. Since $\trop(g_K)$ lies in the support of $\Sigma_1$, the point $g_K$ extends to an $R$-point in $\PP(\Sigma)$, which we denote by $g$. We write $\cX^g$ for the schematic closure of $g_K^{-1}X$ in $\TT$. By definition, the special fiber $(\cX^g)_k$ is isomorphic to the initial degeneration $\mathrm{in}_{\gamma}X$.

 Next, we consider the multiplication morphism $$m:\TT\times_R \cX\to \PP(\Sigma).$$  We denote by $m^{-1}(g)$ the base change of $m$ to $g\in \PP(\Sigma)(R)$. We can consider $m^{-1}(g)$ as a $\TT$-scheme via the composition of the closed embedding $m^{-1}(g)\to \TT\times_R \cX$, the projection $\TT\times_R \cX\to \TT$ and the morphism $\TT\to \TT$ that sends a point of the torus to its multiplicative inverse. Then the structural morphism $m^{-1}(g)\to \TT$ is a proper monomorphism, and thus a closed embedding. It follows that  $m^{-1}(g)$ coincides with the closed subscheme $\cX^g$ of $\TT$, since $m^{-1}(g)$ is flat over $R$ and its generic fiber coincides with $g^{-1}X$.

The point $g_k\in \PP(\Sigma)(k)$ does not depend on the choice of $g$, because $\widetilde{\TT}_k$ acts trivially on the torus orbit of $\PP(\Sigma)_k$ containing $g_k$ (that is, the orbit corresponding to $\gamma$).  If we let $\TT_k$ act on $g_k$ by multiplication, then we can identify this orbit in $\PP(\Sigma)_k$ with $\TT_k/\widetilde{\TT}_k$, and the fiber of $m$ over $g_k$ with $\TT_k\times_{\TT_k/\widetilde{\TT}_k} \cX_k(\gamma)$, which is a trivial $\widetilde{\TT}_k$-torsor over $\cX_k(\gamma)$. Thus, we find that $\cX^g_k$ and $\TT_k\times_{\TT_k/\widetilde{\TT}_k} \cX_k(\gamma)$ are equal as closed subschemes of $\TT_k$. In particular, $\cX^g_k$ does not depend on $g$, and $\cX_k(\gamma)$ is smooth over $k$ if and only if $\mathrm{in}_{\gamma}X$ is smooth over $k$.

\smallskip

{\em Step 2.} Now, assume that $\mathrm{in}_{\gamma}X$ is smooth over $k$. We will cover $\cX_k(\gamma)$ by open subschemes $U$ such that
 $$[\spe^{-1}_{\cX}(U')\cap X(K)]=\Theta([U']_{d-\mathrm{dim}(\gamma)}\otimes [\mathring{\gamma}]_{\mathrm{dim}(\gamma)})\quad
\in K_0(\VF_K)$$ for every open subscheme $U'$ of $U$. Then the statement of the proposition follows by additivity and the fact that $X_\gamma=\spe^{-1}_{\cX}(\cX_k(\gamma))\cap X(K)$.

 Since $\cX_k(\gamma)$ is smooth over $k$ and $\TT/\widetilde{\TT}$ is smooth over $R$, every point of $\cX_k(\gamma)$ has an open neighbourhoud $\cU$ in $\TT/\widetilde{\TT}$ such that there exists an \'etale morphism of $R$-schemes
 $h:\cU\to \AA^r_R$ with $$\cU_k\cap \cX_k(\gamma)=h^{-1}(\AA^{s}_k),$$ where $r=n-\mathrm{dim}(\gamma)$, $s=d-\mathrm{dim}(\gamma)$. Then $\cY=\cU\times_{\AA^r_R}\AA^s_R$ is a closed subscheme of $\cU$, smooth over $R$, with special fiber $\cY_k=\cX_k(\gamma)\cap \cU_k$.

We will construct a semialgebraic bijection
 $$\spe^{-1}_{\cX}(\cY_k)\cap X(K)\to  \widetilde{T}_{\gamma}\times\cY(R)$$
that commutes with the specialization maps to $\cY_k$. This is sufficient to finish the proof, since
 $$[\spe_{\cY}^{-1}(U)]=\Theta([U]_{d-\mathrm{dim}(\gamma)})$$
in $K_0(\VF_K)$ for every open subset $U$ of $\cY_k$, by Proposition \ref{prop:smoothvol}, and
 $$[\widetilde{T}_{\gamma}]=\Theta([\mathring{\gamma}]_{\mathrm{dim}(\gamma)})$$
in $K_0(\VF_K)$ by construction.

  We choose a splitting $T\cong \widetilde{T}\times_K (T/\widetilde{T})$ of the $K$-torus $T$. It induces a projection morphism $p:T\to \widetilde{T}$, as well as a $\widetilde{\TT}$-equivariant isomorphism
  $$\TT\times_{\TT/\widetilde{\TT}}\cU\cong \widetilde{\TT}\times_R \cU$$ that restricts to an isomorphism
  $$\TT\times_{\TT/\widetilde{\TT}}\cY\cong \widetilde{\TT}\times_R \cY.$$
 By the henselian property of $R$, the linear projection  $\pi:\AA^r_R\to \AA^s_R$ lifts to a semialgebraic retraction
$\spe^{-1}_{\cU}(\cY_k)\to \cY(R)$. By base change, this retraction induces a semialgebraic map
 $$\rho:\widetilde{\TT}(R)\times \spe^{-1}_{\cU}(\cY_k)\to \widetilde{\TT}(R)\times \cY(R)$$
   By restricting $\rho$ we obtain, for every point $g_K$ of $\widetilde{T}_\gamma$, a semialgebraic map
 $$\rho_g:(\cX^g\times_{\TT/\widetilde{\TT}}\cU)(R)\to  \widetilde{\TT}(R)\times \cY(R).$$
 This is a bijection: the morphism $$\mathrm{Id}\times (\pi \circ h):\widetilde{\TT}\times_R \cU\to \widetilde{\TT}\times_R\AA^s_R$$ restricts to a morphism
 $$\cX^g\times_{\TT/\widetilde{\TT}}\cU\to \widetilde{\TT}\times_R\AA^s_R$$ which is \'etale because its restriction to the special fibers
 coincides with the \'etale morphism $$\mathrm{Id}\times h_k:\widetilde{\TT}_k\times_k \cY_k\to  \widetilde{\TT}_k\times_k \AA^s_k$$ (see Step 1).
 The henselian property of $R$ now implies that $\rho_g$ is bijective.

     Finally, we consider the map
  $$\psi: \spe^{-1}_{\cX}(\cY_k)\cap X(K)\to  \widetilde{T}_{\gamma}\times \cY(R):x\mapsto (p(x),y)$$
  where $y$ is the image of $\rho(p(x)^{-1}x)$ under the projection to $\cY$. The map $\psi$ is semialgebraic and commutes with specialization to $\cY_k$. Moreover, $\psi$ is bijective: its inverse is given by
  $$(g,y)\mapsto g\ast\rho^{-1}_g(1,y).$$ This concludes the proof.
\end{proof}

\begin{cor}\label{cor:tropvol}
Let $X$ be a reduced  closed subscheme of $T$ and let
 $\tau$ be a definable set in $N\otimes_{\ZZ}G\cong G^n$.
Assume that
 $\mathrm{in}_{w}(X)$ is smooth over $k$ for point $w$ in $\tau$, and that the class $[\mathrm{in}_w X]$ in $K_0(\Var_k)$ does not depend in $w$; we denote it by
 $[\mathrm{in}_\tau X]$.
 Then we have $$\Vol(X(K)\cap \trop^{-1}(\tau))=\chi'(\tau) [\mathrm{in}_{\tau}X]$$
 in $K_0(\Var_k)$.

 In particular, if $\Sigma$ is a tropical fan for $X$ and $\gamma$ is a cell of $\Sigma_1$ such that $\mathrm{in}_\gamma X$ is smooth over $k$, then $\Vol(X_\gamma)=(-1)^{\mathrm{dim}(\gamma)}[\mathrm{in_\gamma}X]$
 if $\gamma$ is bounded, and $\Vol(X_\gamma)=0$ if $\gamma$ is unbounded. If $X$ is sch{\"o}n, then
 $$\Vol(X)=\sum_{\gamma}(-1)^{\mathrm{dim}(\gamma)}[\mathrm{in}_{\gamma}X]\quad
\in K_0(\Var_k)$$
 where the sum is taken over the bounded cells $\gamma$ of $\Sigma_1$.
\end{cor}
\begin{proof}
 We can refine the fan $\Sigma$ in such a way that $\tau$ becomes a union of relative interiors of cells in $\Sigma_1$.
 Thus, by additivity, we may assume that $\tau$ is the relative interior of a cell $\gamma$ of $\Sigma_1$.
 As we have seen in the proof of Proposition \ref{prop:tropvol}, the initial degeneration $\mathrm{in}_{\gamma}X$ is isomorphic to
$\cX_k(\gamma)\times_k \widetilde{\TT}_k$, where $\widetilde{\TT}_k$ is a $k$-torus of dimension equal to $\mathrm{dim}(X)-\mathrm{dim}(\gamma)$.
Now the result follows from Propositions \ref{prop:chi-poly} and \ref{prop:tropvol}, and the definition of the motivic volume (note that $\gamma$ is not an unbounded affine subspace of $N_{\RR}$ because the cones in the tropical fan $\Sigma$ are strictly convex by definition).
\end{proof}

By an additivity argument, we can also obtain an expression for the motivic volume of the schematic closure $\overline{X}$ of $X$ in $\PP(\Sigma)_K$.
 For every cell $\gamma$ of $\Sigma_1$, we set $\overline{X}_\gamma=\spe_{\cX}^{-1}(\cX_k(\gamma))$.
 Equivalently, $\overline{X}_\gamma$ is the closure of $X_\gamma$ in $\overline{X}(K)$ with respect to the valuation topology. These sets form a semialgebraic partition of $\overline{X}(K)$. We have $X_\gamma=X(K)\cap \overline{X}_\gamma$, and $X_\gamma=\overline{X}_\gamma$ if and only if $\gamma$ is bounded.

\begin{prop}\label{prop:tropvol2}
Let $X$ be a reduced  closed subscheme of $T$, let $\Sigma$ be a tropical fan for $X$ and denote by $\overline{X}$ the schematic closure of $X$ in $\PP(\Sigma)_K$.  Then for every cell $\gamma$ of $\Sigma_1$ such that $\mathrm{in}_{\gamma}(X)$ is smooth over $k$, we have
$$\Vol(\overline{X}_\gamma)=(1-\LL)^{\mathrm{dim}(\gamma)-\mathrm{dim}(\rec(\gamma))}[\cX_k(\gamma)]$$
 where $\rec(\gamma)$ denotes the recession cone of $\gamma$. In particular, if $X$ is sch{\"o}n, then
 $$\Vol(\overline{X})=\sum_{\gamma\in \Sigma_1}(1-\LL)^{\mathrm{dim}(\gamma)-\mathrm{dim}(\rec(\gamma))}[\cX_k(\gamma)]\quad
\in K_0(\Var_k).$$
\end{prop}

\begin{proof}
We denote by $\rec(\Sigma)$ the recession fan of $\Sigma$ in $N_\RR$. It consists of the  recession cones $\rec(\gamma)$ of the cells $\gamma$ of $\Sigma_1$. The toric variety associated with $\rec(\Sigma)$ is canonically isomorphic with $\PP(\Sigma)_K$. Let $\sigma$ be a cone of $\rec(\Sigma)$. We denote by $T(\sigma)$ the torus orbit of $\PP(\Sigma)_K$ corresponding to $\sigma$.  We can  view $T(\sigma)$ as a quotient $K$-torus of $T$ of dimension $n-\mathrm{dim}(\sigma)$ with cocharacter lattice  $$N(\sigma)=N/(N\cap V_{\sigma}),$$ where $V_{\sigma}$ is the subspace of $N_\RR$ spanned by $\sigma$.  We denote by $p$ the projection morphism $$p:N_\RR\to N(\sigma)_\RR,$$ and we write $\TT(\sigma)$ for the split $R$-torus with the same cocharacter lattice as $T(\sigma)$.

Let $S(\sigma)$ be the set of cones in $\Sigma$ whose recession cones contain $\sigma$. By projecting the cones in $S(\sigma)$ to $N(\sigma)_{\RR}\oplus \RR$, we get a $G$-admissible fan that we denote by $\Sigma(\sigma)$.  The associated toric $R$-scheme $\PP(\Sigma(\sigma))$ is canonically isomorphic to the schematic closure of $T(\sigma)$ in $\PP(\Sigma)$. We denote by $X(\sigma)$ the intersection of $\overline{X}$ with $T(\sigma)$, with its reduced induced structure, and by $\cX(\sigma)$ the schematic closure of $X(\sigma)$ in $\PP(\Sigma(\sigma))$. Then the multiplication morphism
  $$m':\cX(\sigma)\times_R \TT\to \PP(\Sigma(\sigma))$$
is the base change of the multiplication morphism
  $$m:\cX\times_R \TT\to \PP(\Sigma).$$
Indeed, this is obviously true over the dense torus orbit of $\PP(\Sigma(\sigma))_K$, but then it holds over the whole of $\PP(\Sigma(\sigma))$ by flatness of $m$.  Since the action of $\TT$ on $\PP(\Sigma(\sigma))$ factors through the quotient torus $\TT(\sigma)$, the morphism $m'$ factors through the multiplication morphism
     $$m_{\sigma}:\cX(\sigma)\times_R \TT(\sigma)\to \PP(\Sigma(\sigma)).$$
It follows that $m_{\sigma}$ is still faithfully flat, so that $\Sigma(\sigma)$ is a tropical fan for $X(\sigma)$.

  Our description of $m_{\sigma}$ also implies that, for every cell $\gamma$ of $\Sigma_1$ whose recession cone contains $\sigma$, the stratum $\cX_k(\gamma)$ of the special fiber of $\cX$ coincides with the stratum $\cX(\sigma)_k(p(\gamma))$ of the special fiber of $\cX(\sigma)$. It then follows from Step 1 in the proof of Proposition \ref{prop:tropvol} that $\mathrm{in}_{\gamma}(X)$ is isomorphic to  $\mathrm{in}_{p(\gamma)}(X(\sigma))\times_k \GG^{\mathrm{dim}(\sigma)}_{m,k}$.
  In particular, $\mathrm{in}_{\gamma}(X)$ is smooth if and only if $\mathrm{in}_{p(\gamma)}(X(\sigma))$ is smooth.
    Moreover, $p(\gamma)$ is bounded if and only if  $\mathrm{rec}(\gamma)=\sigma$; in that case, the dimension of $p(\gamma)$ is equal to $\mathrm{dim}(\gamma)-\mathrm{dim}(\rec(\gamma))$. Since we can write $$[\overline{X}_\gamma]=\sum_{\stackrel{\sigma\in \rec(\Sigma)}{\sigma\subset \rec(\gamma)}}[X(\sigma)_{p(\gamma)}]$$ in $K_0(\VF_K)$, the formula in Corollary \ref{cor:tropvol} now yields $$\Vol(\overline{X}_\gamma)=(1-\LL)^{\mathrm{dim}(\gamma)-\mathrm{dim}(\rec(\gamma))}[\cX_k(\gamma)]\quad \in K_0(\Var_k)$$ whenever $\mathrm{in}_{\gamma}(X)$ is smooth over $k$.
  The expression for $\Vol(\overline{X})$ then follows by additivity.
\end{proof}
\begin{rem}
Note that, in the statement of Proposition \ref{prop:tropvol2}, the recession cone of $\gamma$ has dimension zero if $\gamma$ is bounded, so that we get the same formula for the motivic volume of $\overline{X}_{\gamma}=X_\gamma$ as in Corollary \ref{cor:tropvol} for bounded cells $\gamma$.
\end{rem}

We also record the following variant of Proposition \ref{prop:tropvol2} that will be used in the calculations in Section \ref{sec:genus1}. Let us emphasize that, in the statement of Proposition \ref{prop:tropvol3}, we do {\em not} assume that $\Sigma$ is a tropical fan for $X$.

\begin{prop}\label{prop:tropvol3}
Let $X$ be an integral closed subscheme of $T$, let $\Sigma$ be a $G$-admissible fan in $N_{\RR}\oplus \RR_{\geq 0}$, and denote by $\overline{X}$ the schematic closure of $X$ in $\PP(\Sigma)_K$.
 Let $\tau$ be a $G$-rational polyhedron in $N_{\RR}$ whose recession cone $\mathrm{rec}(\tau)$ belongs to the recession fan $\mathrm{rec}(\Sigma)$.
  Assume that $\mathrm{in}_{w}X$ is smooth over $k$ for every $G$-rational point $w$ in $\mathring{\tau}$, and that its class in $K_0(\Var_k)$ does not depend on $w$; we denote it by $[\mathrm{in}_\tau X]$.

  We denote by $\overline{X}_\tau$ the closure of $X(K)\cap \trop^{-1}(\mathring{\tau})$ in $\overline{X}(K)$ with respect to the valuation topology.
    Then  we have
$$\Vol(\overline{X}_\tau)=(-1)^{\mathrm{dim}(\tau)-\mathrm{dim}(\rec(\tau))}\frac{[\mathrm{in}_{\tau}X]}{(\LL-1)^{\mathrm{dim}(\rec(\tau))}}$$
in $K_0(\Var_k)[(\LL-1)^{-1}]$.
\end{prop}
\begin{proof}
The argument is similar to the proof of Proposition \ref{prop:tropvol2}; we adopt the notations of that proof. We may assume that every $G$-rational point $w$ of $\mathring{\tau}$ is contained in $\trop(X(K))$, since otherwise, all the initial degenerations $\mathrm{in}_{w}X$ are empty, and there is nothing to prove.
 Let $\sigma$ be a cone in $\mathrm{rec}(\Sigma)$. Let $\tau'$ be the set of points $w'$ in $N(\sigma)_{\RR}$ such that $w+\sigma\subset \tau$ for some $w$ in $p^{-1}(w')$.
 Then $\tau'$ is a $G$-rational polyhedron in $N(\sigma)_{\RR}$, and its relative interior is equal to $p(\mathring{\tau})$. Since we are assuming that the recession cone of $\tau$ belongs to the fan $\rec(\Sigma)$, the polyhedron $\tau'$ is empty unless $\sigma$ is a face of $\mathrm{rec}(\tau)$;
 it is bounded of dimension $\mathrm{dim}(\tau)-\mathrm{dim}(\mathrm{rec}(\tau))$  if $\sigma$ coincides with $\mathrm{rec}(\tau)$, and it is unbounded if $\sigma$ is a strict face of $\mathrm{rec}(\tau)$.
 Moreover,  we have $$\overline{X}_\tau\cap T(\sigma)(K)=X(\sigma)(K)\cap \trop^{-1}(\mathring{\tau}').$$
  Thus it suffices to show that, if $\sigma$ is a face of $\mathrm{rec}(\tau)$ and $w'$ is a $G$-rational point in $\mathring{\tau}'$, then there exists a $G$-rational point $w$ in $\mathring{\tau}\cap p^{-1}(w')$ such that the initial degeneration $\mathrm{in}_w X$ is isomorphic to
 $\mathrm{in}_{p(w)}X(\sigma)\times_k \mathbb{G}_{m,k}^{\mathrm{dim}(\sigma)}$. Smoothness of $\mathrm{in}_w X$ then implies smoothness of $\mathrm{in}_{p(w)}X(\sigma)$, and the result follows from Corollary \ref{cor:tropvol} and the additivity of the motivic volume.

 So assume that $\sigma$ is a face of $\mathrm{rec}(\tau)$, and let $w'$ be a $G$-rational point in $\mathring{\tau}'$.
 Let $\Sigma'$ be a tropical fan in $N_{\RR}\oplus \RR_{\geq 0}$ for $X$ such that $\tau$ is a union of cells in $\Sigma'_1$.
 Then we can find:
 \begin{enumerate}
\item a cell $\gamma$ in $\Sigma'_1$ contained in $\tau$;
\item  a $G$-rational point $w$  in the relative interior of $\gamma$ such that $p(w)=w'$;
\item a face $\sigma'$ of the recession cone of $\gamma$ such that $\sigma'$ is contained in $\sigma$ and of the same dimension as $\sigma$.
\end{enumerate}
   We denote by $T(\sigma')$ the torus orbit of $\PP(\Sigma')_K$ corresponding to $\sigma'$.
      The inclusion of $\sigma'$ in $\sigma$ induces an isomorphism of tori $T(\sigma')\to T(\sigma)$.
  Let $\mathcal{X}'$ be the closure of $X$ in $\mathbb{P}(\Sigma')$, and denote by $X(\sigma')$ the intersection of $\mathcal{X}'_K$ with $T(\sigma')$, with its reduced induced structure. Since $\Sigma'$ is a tropical fan for $X$, the multiplication morphism $\mathcal{X}'\times \mathbb{T}\to \mathbb{P}(\Sigma')$ is faithfully flat. This property is preserved by base change to $K$; thus the recession fan of $\Sigma'$ is a tropical fan for $X$ with respect to the trivial absolute value on $K$.
 We can choose $\Sigma'$ in such a way that $\rec(\Sigma')$ is a refinement of a fan that contains $\sigma$ as a cone.
    Then it follows from \cite[4.4]{LQ} that $X(\sigma')$ is
 equal to
    the inverse image of $X(\sigma)$ in $T(\sigma')$.
   Therefore, the initial degeneration of $X(\sigma')$ at $w'$ is isomorphic to $\mathrm{in}_{w'}X(\sigma)$.
 On the other hand, since $\Sigma'$ is a tropical fan for $X$, the proof of Proposition \ref{prop:tropvol2} shows that $\mathrm{in}_{w} X$ is isomorphic to $\mathrm{in}_{w'}X(\sigma')\times \mathbb{G}_{m,k}^{\dim(\sigma')}$. This concludes the proof.
\end{proof}

\subsection{Comparison with the motivic nearby fiber}
Another situation where we can explicitly compute the motivic volume is the following. We say that a flat $R$-scheme of finite type is {\em strictly semi-stable} if it can be covered with open subschemes that admit an \'etale morphism to an $R$-scheme of the form
 $$S_{d,r,a}=\Spec R[x_0,\ldots,x_d]/(x_0\cdot \ldots \cdot x_r-a)$$
where $r\leq d$ and $a$ is a non-zero element of the maximal ideal of $R$.  Let $\cX$ be a strictly semi-stable $R$-scheme of pure relative dimension $d$. We denote by $E_i,\,i\in I$ the irreducible components of $\cX_k$. For every non-empty subset $J$ of $I$, we set
  $$E_J=\bigcap_{j\in J} E_j,\quad E_J^o=E_J\smallsetminus \left(\bigcup_{i\notin J}E_i\right).$$
The sets $E_J^o$ form a stratification of $\cX_k$ into locally closed subsets.

\begin{prop}\label{prop:sstablevol}
For every non-empty subset $J$ of $I$, we have
  $$[\spe_{\cX}^{-1}(E_J^o)]=(-1)^{|J|-1}\Theta([E_J^o]_{d+1-|J|}\otimes [1]_{|J|-1})$$
in $K_0(\VF_K)$. In particular,
  $$[\cX(R)]=\sum_{\emptyset \neq J\subset I}(-1)^{|J|-1}\Theta([E_J^o]_{d+1-|J|}\otimes [1]_{|J|-1})$$
in $K_0(\VF_K)$.
\end{prop}

\begin{proof}
By additivity, it suffices to prove the first assertion, and we may assume that
there exists an \'etale morphism
 $$h:\cX\to S_{d,r,a}=\Spec R[x_0,\ldots,x_d]/(x_0\cdot \ldots \cdot x_r-a),$$
with $r=|J|-1$ and $a$ a non-zero element of the maximal ideal of $R$, such that $E_J^o$ is the inverse image under $h$ of the zero locus of $(x_0,\ldots,x_r)$ in the special fiber of $S_{d,r,a}$. We write
$$S_{d,r,a}=S_{r,r,a}\times_R \AA^{d-r}_R$$ and we denote by $O$ the origin of the special fiber of $S_{r,r,a}$.
 We choose a point $y$ in $\spe_{S_{r,r,a}}^{-1}(O)$ and we denote by $\cY$ the  inverse image of $\{y\}\times \AA^{d-r}_R$ under $h$. Then $\cY$ is an \'etale  $\AA^{d-r}_R$-scheme with special fiber $E_J^o$. Now we consider the map
\begin{equation}\label{eq:semi-alg}
\spe_{\cX}^{-1}(E_J^o)\to \spe_{S_{r,r,a}}^{-1}(O)\times \cY(R):u\mapsto (v_1(u),v_2(u))
\end{equation}
where $v_1(u)$ is the projection of $h(u)\in S_{d,r,a}(R)$ onto $S_{r,r,a}(R)$ and $v_2(u)$ is the unique point in $\cY(R)$ such that $\spe_{\cX}(u)=\spe_{\cY}(v_2(u))$, and such that $h(v_2(u))$ is the projection of $h(u)$ onto $\AA^{d-r}_R(R)$. It is clear from the construction that the map \eqref{eq:semi-alg} is a semialgebraic bijection, and thus an isomorphism in the category $\VF_K$. The class of $\cY(R)$ in $K_0(\VF_K)$ is precisely $\Theta([E_J^o]_{d+1-|J|})$, by Proposition \ref{prop:smoothvol}.  By projection on the last $r$ coordinates, we can identify the semialgebraic set $\spe_{S_{r,r,a}}^{-1}(O)$ with $$ \{x\in (K^{\ast})^{r}\mid \sum_{i=1}^r \val(x_i)<\val(a)\mbox{ and }\val(x_i)>0\mbox{ for all }i\}.$$
Thus, with the notation from Proposition \ref{prop:scale},
  $$[\spe_{S_{r,r,a}}^{-1}(O)]=\Theta([\Delta^o_{r,\val(a)}]_r)=(-1)^r\Theta([1]_r)$$
 in $K_0(\VF_K)$ by the definition of the morphism $\Theta$.
 The result now follows from the multiplicativity of $\Theta$.
\end{proof}

\begin{cor}\label{cor:sstablecor}
With the notations of Proposition \ref{prop:sstablevol}, we have
$$\Vol(\cX(R))=\sum_{\emptyset \neq J\subset I}[E_J^o](1-\LL)^{|J|-1}$$
in $K_0(\Var_k)$.
\end{cor}

\begin{proof}
This follows at once from Proposition \ref{prop:sstablevol}.
\end{proof}

Using Corollary \ref{cor:sstablecor}, we can compare the motivic volume to other motivic invariants that appear in the literature: the motivic nearby fiber of Denef--Loeser \cite[3.5.3]{DL} and the motivic volume of smooth rigid varieties defined by Sebag and the first-named author \cite[8.3]{NiSe}. The motivic nearby fiber was defined as a motivic incarnation of the complex of nearby cycles associated with a flat and generically smooth morphism of $k$-varieties $f:Z\to \AA^1_k$, and the motivic volume of a smooth rigid variety extends this construction to formal schemes over $k\llbracket t\rrbracket$.
\begin{cor}\label{cor:compar} We assume that $K$ is an algebraically closed valued field extension of $k(\!(t)\!)$.

\begin{enumerate}

\item Let $f:Z\to \Spec k[t]$ be a flat and generically smooth morphism of $k$-schemes of finite type, and denote by $\cZ$ the base change of $Z$ to the valuation ring $R$ of $K$.
    Then the image of $\Vol(\cZ(R))$ in the localized Grothendieck ring $\mathcal{M}_k=K_0(\Var_k)[\LL^{-1}]$ is equal to Denef and Loeser's motivic nearby fiber of $f$ (forgetting the $\widehat{\mu}$-action).

\item Let $\cX$ be a generically smooth flat $k\llbracket t\rrbracket$-scheme of finite type of pure relative dimension $d$, and denote by $\widehat{\cX}$ its formal $t$-adic completion. Then the image of $\LL^{-d}\Vol(\cX(R))$ in $\mathcal{M}_k$ is equal to the motivic volume of the generic fiber of  $\widehat{\cX}$ (which is a quasi-compact smooth rigid $k(\!(t)\!)$-variety).
\end{enumerate}

\end{cor}

\begin{proof}
The first assertion is a special case of the second, by the comparison result in \cite[9.13]{NiSe}. The second assertion follows from the explicit formula in \cite[6.11]{Ni-sing}.
\end{proof}

This result shows, in particular, that the motivic nearby fiber is well-defined as an element of $K_0(\Var_k)$, without inverting $\LL$. This is not at all obvious: inverting $\LL$ is an essential step in the definition of the motivic zeta function (which is used to construct the motivic nearby fiber), and it was recently proved by Borisov that $\LL$ is a zero divisor in $K_0(\Var_k)$ \cite{Bo}.

\begin{rem}
 Combining Corollary \ref{cor:tropvol}, Proposition \ref{prop:tropvol2} and Corollary \ref{cor:compar}, we recover the formulas of Katz and Stapledon for motivic nearby fibers of sch\"{o}n varieties over the field of meromorphic germs at the origin of the complex plane; see \cite[5.1]{KS1} and \cite[2.4]{KS14}. Their method was entirely different: they first showed that the desired tropical formula for the motivic nearby fiber was independent of the choice of a tropical fan, and used this to reduce to the case where $\cX$ is a strictly semi-stable model, where one can use the explicit formula for the motivic nearby fiber. Corollary \ref{cor:compar} is also closely related to similar comparison results by Hrushovski and Loeser for the motivic zeta function \cite{HL15}, but our approach is more direct if one only wants to retrieve the motivic nearby fiber; in particular, we avoid inverting $\LL$.
 \end{rem}

\subsection{The $\chi_{-y}$-genus of a semialgebraic set}\label{ss:chiy}
 We can use the motivic volume to define the limit $\chi_{-y}$-genus of a semialgebraic set over $K$.
 Recall that, for every field $F$ of characteristic zero, there exists a unique ring morphism
 $$\chi_{-y}:K_0(\Var_F)\to \ZZ[y]$$ that maps the class of every smooth and proper $F$-scheme $Z$ to $$\chi_{-y}(Z)=\sum_{p,q\geq 0}(-1)^{p+q}h^{p,q}(Z)y^q=\sum_{q\geq 0}(-1)^q\chi(Z,\Omega^{q}_Z)y^q.$$
 For every $F$-scheme of finite type $Z$, we denote by $\chi_{-y}(Z)$ the image of $[Z]$ under $\chi_{-y}$. This invariant is called the $\chi_{-y}$-genus of $Z$.

 \begin{defn}\label{def:chiylim}
 We define the limit $\chi_{-y}$-genus of a semialgebraic set $S$ over $K$ by
 $$\chi^{\lim}_{-y}(S)=\chi_{-y}(\Vol(S))\quad \in \ZZ[y].$$
 \end{defn}

 By specializing our tropical formulas for the motivic volume (Corollary \ref{cor:tropvol} and Proposition \ref{prop:tropvol2}) with respect to $\chi_{-y}$, we immediately obtain the following expressions.

 \begin{prop}\label{prop:tropchiy}
Let $X$ be a sch\"{o}n reduced closed subscheme of $T$ and let $\Sigma$ be a tropical fan for $X$. Denote by $\overline{X}$ the schematic closure of $X$ in $\PP(\Sigma)_K$.
 Then
 $$\chi^{\lim}_{-y}(X(K))=\sum_{\gamma}(-1)^{\mathrm{dim}(\gamma)}\chi_{-y}(\mathrm{in}_{\gamma}X)\quad
\in \ZZ[y]$$
where the sum is taken over the bounded cells $\gamma$ of $\Sigma_1$. Moreover,
 \begin{eqnarray*}
 \chi^{\lim}_{-y}(\overline{X}(K))&=&\sum_{\gamma\in \Sigma_1}(1-y)^{\mathrm{dim}(\gamma)-\mathrm{dim}(\rec(\gamma))}\chi_{-y}(\cX_k(\gamma))
 \\ &=&\sum_{\gamma\in \Sigma_1}(-1)^{\mathrm{dim}(\gamma)-\mathrm{dim}(\rec(\gamma))}(y-1)^{-\mathrm{dim}(\rec(\gamma))}\chi_{-y}(\mathrm{in}_{\gamma}X)
\end{eqnarray*}
in $\ZZ[y]$.
\end{prop}

\begin{proof}
This follows at once from Corollary \ref{cor:tropvol} and Proposition \ref{prop:tropvol3}.
\end{proof}

 For every $K$-scheme of finite type $X$, we can consider both its $\chi_{-y}$-genus $\chi_{-y}(X)$ and the limit $\chi_{-y}$-genus $\chi_{-y}^{\lim}(X(K))$. We expect that these invariants always coincide; unfortunately, we do not know how to prove this without assuming a suitable form a resolution of singularities for schemes over $R$. Therefore, we limit ourselves to the following partial result, which is sufficient for the applications in this paper.

\begin{prop}\label{prop:limnolim}
Assume that $K$ is an algebraic closure of a henselian discretely valued field. Then for every algebraic $K$-scheme of finite type $X$, we have
$$\chi_{-y}^{\lim}(X(K))=\chi_{-y}(X).$$
\end{prop}

\begin{proof}
By our assumption on $K$, the scheme $X$ is defined over some  henselian discretely valued subfield $K'$ of $K$.    We denote by $R'$ and $k'$ the valuation ring and residue field of $K'$, respectively.
 By additivity, we may assume that $X$ is smooth and proper over $K'$. Passing to a finite extension of $K'$ if necessary, we can moreover suppose that $X$ has a regular strictly semi-stable model $\cX$ over $R'$. Recall that this means that $\cX$ is regular flat proper $R'$-model and that its special fiber $\cX_{k'}$ is a reduced strict normal crossings divisor. Then writing $$\cX_{k'}=\sum_{i\in I}E_i,$$ we can use Corollary \ref{cor:sstablecor} to compute the motivic volume of $X$, and we find that
 $$\chi_{-y}^{\lim}(X(K))=\sum_{\emptyset \neq J\subset I}\chi_{-y}(E_J^o)(1-y)^{|J|-1}$$
(using the notation of Corollary \ref{cor:sstablecor}). We need to show that this expression is also equal to $\chi_{-y}(X)$.

 We will apply Hodge theory for logarithmic $R'$-schemes; see \cite[\S7]{IKN} and \cite[\S2.2]{SV}. The weight spectral sequence for de Rham cohomology of the model $\cX$ has $E_1$-sheet $$E_1^{pq}=\bigoplus_{i,i-p\geq 0}\bigoplus_{\stackrel{J\subset I}{ |J|=2i-p+1}} H^{q+2p - 2i}_{\mathrm{dR}}(E_J)(p-i)\Rightarrow H^{p+q}_{dR}(\cX^+_{k'}),$$ where $\cX^+_{k'}$ is the special fiber of $\cX$ with its induced log structure (this is a proper log smooth log scheme over the standard log point $(\Spec k')^+$). This weight spectral sequence degenerates at $E_2$, and it is compatible with the respective Hodge filtrations. The flags of the Hodge filtration on $H^{p+q}_{dR}(\cX^+_{k'})$ have the same dimension as those on $H^{p+q}_{dR}(X)$, by \cite[7.2]{IKN}. It follows that the $\chi_{-y}$-genus of $X$ is equal to $$\sum_{i,i-p\geq 0}\sum_{\stackrel{J\subset I}{ |J|=2i-p+1}}\chi_{-y}(E_J)(-1)^py^{i-p}=\sum_{\emptyset \neq J\subset I}\chi_{-y}(E_J^o)(1-y)^{|J|-1}$$ as required (here we used that the sets $E^o_{J'}$ with $J\subset J'$ form a partition of $E_J$).
\end{proof}

\begin{rem}
Our proof shows, in fact, that the Hodge--Deligne polynomial of $\Vol(X)$ is equal to the Hodge-Deligne polynomial of the limit mixed Hodge structure associated with $X$, but we will not need this property.
\end{rem}

\begin{defn}\label{def:chiy}
If $K$ is as in Proposition \ref{prop:limnolim} (for instance, $K=\Puis$) and $S$ is a semialgebraic set over $K$, we will write $\chi_{-y}(S)$ instead of $\chi^{\lim}_{-y}(S)$, and we simply call this invariant the $\chi_{-y}$-genus of $S$.
\end{defn}

Proposition \ref{prop:limnolim} guarantees that this definition does not lead to ambiguities.

\subsection{The Euler characteristic of a semialgebraic set}  \label{sec:Euler}

 To conclude this section, we establish some properties of the specialization of the motivic volume with respect to the Euler characteristic, and we compare it to the Euler characteristic of Berkovich's \'etale cohomology for nonarchimedean analytic spaces \cite{berk-etale}. We endow $K$ with the nonarchimedean absolute value given by $|x|=\exp(-\val(x))$ for every $x\in K^{*}$, and we denote by $\widehat{K}$ the completion of $K$. For every $K$-scheme of finite type $X$, we denote by $X^{\an}$ the $\widehat{K}$-analytic space associated with $X\times_K \widehat{K}$.

 If $Y$ is a $K$-analytic space and $T$ is a subset of $Y$, then the {\em germ} $(Y,T)$ of $Y$ at $T$ is defined in \cite[\S3.4]{berk-etale}.  If $f:Y'\to Y$ is an isomorphism from $Y'$ onto an open subspace of $Y$ containing $T$, and $T'=f^{-1}(T)$, then the morphism of germs $(Y',T')\to (Y,T)$ induced by $f$ is declared to be an isomorphism in the category of germs. The \'etale topology on a germ $(Y,T)$ is defined in \cite[\S4.2]{berk-etale}.  We say that $T$ is an analytic subspace of $Y$ if there exist a $K$-analytic space $Z$ and a morphism $f:Z\to Y$ such that $f$ is a homeomorphism onto its image, $f(Z)=T$, and the induced morphism of residue fields $\mathscr{H}(f(z))\to \mathscr{H}(z)$ is an isomorphism for every $z\in Z$. Such a morphism is called a quasi-immersion \cite[4.3.3]{berk-etale}, and it induces an equivalence between the \'etale topoi of the germ $(Y,T)$ and the $K$-analytic space $Z$ by \cite[4.3.4]{berk-etale}. Thus $(Y,T)$ and $Z$ have the same \'etale cohomology spaces.

 Now let $X$ be a $K$-scheme of finite type. To every semialgebraic subset $S$ of $X$, one can attach a subset $S^{\an}$ of $X^{\an}$ that is defined by the same formulas as $S$; see \S5.2 in \cite{HL15}. Subsets of $X^{\an}$ of this form will again be called semialgebraic. If $S^{\an}$ is locally closed in $X^{\an}$, we will call $S$ a locally closed semialgebraic subset of $X$. Then the germ $(X^{\an},S^{\an})$ has finite dimensional $\ell$-adic cohomology spaces with compact supports concentrated in degrees $\leq 2~\mathrm{dim}(X)$, by \cite[5.14]{Mar14}, so that we can consider its $\ell$-adic  Euler characteristic $$\eu(S^{\an})=\sum_{i\geq 0}(-1)^i\mathrm{dim}\,H^i_c((X^{\an},S^{\an})_{\mathrm{\acute{e}t}},\QQ_\ell).$$ It follows from Proposition~5.2.2 in \cite{HL15} that $\eu(S^{\an})$ only depends on
 the isomorphism class of the semialgebraic set $S$, which justifies the omission of the ambient space $X$ from the notation.

For every field $F$, there exists a unique ring morphism $$\eu:K_0(\Var_F)\to \ZZ$$ that sends the class $[Y]$ of every $F$-scheme of finite type $Y$ to $\eu(Y)$, the $\ell$-adic Euler characteristic with compact supports of $Y$, where $\ell$ is any prime number invertible in $F$. If $F$ has characteristic zero, then $\eu(Y)$ is the value of $\chi_{-y}(Y)$ at $y=1$. If $F$ is a subfield of $\CC$, then $\eu(Y)$ is equal to the singular Euler characteristic with compact supports of $Y(\CC)$ with respect to its complex analytic topology.

\begin{prop}\label{prop:analytic}
Let $X$ be a $K$-scheme of finite type and let $S$ be a locally closed semialgebraic subset of $X$.  Then $\eu(\Vol(S))=\eu(S^{\an})$.
\end{prop}

\begin{proof}
By Proposition \ref{prop:complete}, we may assume that $K$ is complete. By  \cite[Proposition~5.2.2]{HL15}, there exists a unique ring morphism
    $$\varepsilon:K_0(\VF_K)\to \ZZ$$
that maps $[S]$ to $\eu(S^{\an})$ for every locally closed semialgebraic set $S$. We will show that $\varepsilon=\eu\circ \Vol$. This equality can be tested on the elements of $K_0(\VF_K)$ of the form $\Theta([\gamma]_n)$,  with $\gamma$ a closed polyhedron in $G^n$, and $\Theta([X]_n)$, with $X$ a  subscheme of the special fiber of a smooth $R$-scheme $\cX$ of relative dimension $n$. Indeed, these elements generate the Grothendieck ring $K_0(\VF_K)$ by the results in \S\ref{sec:KVF}.

Since the Euler characteristic of the torus $\mathbb{G}_{m,k}$ vanishes,
   $$(\eu\circ \Vol)(\Theta([\gamma]_n))=\eu((\LL-1)^n)=0$$
whenever $n>0$. On the other hand, since $\trop^{-1}(\gamma)$ is a closed semialgebraic subset of $K^n$, we have
   $$\varepsilon(\Theta([\gamma]_n))=\eu((\trop^{-1}(\gamma))^{\an})$$
which also vanishes by Lemma 5.4.2 in \cite{HL15}. As for $\Theta([X]_n)$, we have
  $$(\eu\circ \Vol)(\Theta([X]_n))=\eu(X)$$
by the definition of the motivic volume, and
  $$\varepsilon(\Theta([X]_n))=\eu((\spe^{-1}_{\cX}(X))^{\an}).$$
The $K$-analytic space $(\spe^{-1}_{\cX}(X))^{\an}$ is precisely the inverse image of $X$ under the specialization map $\widehat{\cX}_K\to \widehat{\cX}$, where $\widehat{\cX}$ denotes the formal completion of $\cX$ and $\widehat{\cX}_K$ denotes its generic fiber in the category of $K$-analytic spaces. The equality
   $$\eu(X)=\eu((\spe^{-1}_{\cX}(X))^{\an}) $$
 now follows from Berkovich's theory of nearby cycles for formal schemes: see \cite[Lemma~5.4.3]{HL15}.
 \end{proof}

\begin{cor}\label{cor:eu}
If $X$ is a $K$-scheme of finite type, then
 $$\eu(\Vol(X))=\eu(X).$$
\end{cor}

\begin{proof}
This follows at once from Proposition \ref{prop:analytic} and the comparison theorem for \'etale cohomology with compact supports for analytifications of $\widehat{K}$-schemes of finite type \cite[7.1.1]{berk-etale}.
\end{proof}

Thanks to Corollary \ref{cor:eu}, the following definition is  unambiguous.

\begin{defn}\label{def:eu}
We define the Euler characteristic of a semialgebraic subset $S$ over $K$  by
 $$\eu(S)=\eu(\Vol(S)).$$
\end{defn}

Using the comparison result in Proposition \ref{prop:analytic}, we can show that the Euler characteristic of a semialgebraic set satisfies some of the standard cohomological properties. In particular, we will prove that the Euler characteristic of a proper family over a semialgebraic base can be computed by integrating the Euler characteristics of the fibers over the base (Corollary \ref{cor:multiplicative}). This will be essential for the applications in Section \ref{sec:geominterpr}.

 \begin{prop}\label{prop:eulisse}
 Assume that $K=\Puis$. Let $X$ be a $K$-scheme of finite type, and let $\mathscr{F}$ be a constructible sheaf of $\mathbb{F}_\ell$-vector spaces on $X$.  Let $S$ be a locally closed semialgebraic subset of $X$.  We denote the pullback of $\mathscr{F}$ to the germ $(X^{\an},S^{\an})$ again by $\mathscr{F}$. Then the \'etale cohomology spaces
   $$H^i_c((X^{\an},S^{\an})_{\mathrm{\acute{e}t}},\mathscr{F})$$
are finite dimensional for all $i\geq 0$, and vanish for $i>2 \mathrm{dim}(X)$. Moreover, if $\mathscr{F}$ is lisse of rank $n$ on $X$, then
\begin{equation}\label{eq:lisse}
\sum_{i\geq 0}(-1)^i \mathrm{dim}\,H^i_c((X^{\an},S^{\an})_{\mathrm{\acute{e}t}},\mathscr{F})=n \cdot \eu(S^{\an}).\end{equation}
\end{prop}

\begin{proof}
The finiteness and vanishing of the cohomology spaces can be proven in exactly the same way as Theorem 5.14 in \cite{Mar14}, using Berkovich's finiteness result in \cite[1.1.1]{berk-finite}. So let us assume that $\mathscr{F}$ is lisse of rank $n$ on $X$, and prove Equation \ref{eq:lisse}. If $S$ is a subscheme of $X$, this equality was proven by Deligne for algebraic \'etale cohomology \cite[2.7]{illusie}.  We will adapt the proof of \cite[2.7]{illusie} to $K$-analytic spaces.

To start with, we observe that
     $$\eu(S^{\an})=\sum_{i\geq 0}(-1)^i\mathrm{dim}\,H^i_c((X^{\an},S^{\an})_{\mathrm{\acute{e}t}},\mathbb{F}_\ell)$$
because the complex $R\Gamma_c((X^{\an},S^{\an})_{\mathrm{\acute{e}t}},\ZZ_\ell)$ is perfect, by \cite[5.10]{Mar14}. Now let us go through the different steps of the proof of \cite[2.1]{illusie} (from which \cite[2.7]{illusie} immediately follows), and check that they apply to our set-up, as well. Let $Y\to X$ be a connected finite Galois covering with Galois group $G$ such that the pullback of  $\mathscr{F}$ to $Y$ is trivial. We denote by $T$ the inverse image of $S$ in $Y$; this is a locally closed semialgebraic subset of $Y$. The morphism of germs $f:(Y^{\an},T^{\an})\to (X^{\an},S^{\an})$ is still a Galois cover with Galois group $G$.

The complex of $\mathbb{F}_\ell[G]$-modules    $$R\Gamma_c((Y^{\an},T^{\an})_{\mathrm{\acute{e}t}},\mathbb{F}_\ell)\cong
R\Gamma_c((X^{\an},S^{\an})_{\mathrm{\acute{e}t}},f_*\mathbb{F}_\ell)$$ is perfect  (the proof of \cite[5.3.10]{berk-etale} also applies to $\mathbb{F}_\ell[G]$-coefficients, so that we can use the same arguments as in \cite[5.10]{Mar14}). Moreover,
    $$R\Gamma_c((X^{\an},S^{\an})_{\mathrm{\acute{e}t}},\mathscr{F})\cong
    R\Gamma^G(R\Gamma_c((Y^{\an},T^{\an})_{\mathrm{\acute{e}t}},\mathbb{F}_\ell)\otimes_{\mathbb{F}_\ell}\mathscr{F}_x)$$
where $x$ is any point of $S(K)$ and $G$ acts diagonally on the tensor product in the right hand side (by the same arguments as in \cite{illusie}). Thus we can use formula (2.3.1) in \cite{illusie} to compute the left hand side of \eqref{eq:lisse}. Now it suffices to show that, for every element $g\neq 1$ in $G$, the trace of $g$ on $R\Gamma_c((Y^{\an},T^{\an})_{\mathrm{\acute{e}t}},\QQ_\ell)$ vanishes. This is automatic when the order of $g$ in $G$ is divisible by $\ell$, by \cite[III.3.2]{serre}. Hence, it is enough to prove that the trace of $g$ lies in $\ZZ$ and is independent of $\ell$.

By an additivity argument, we may assume that $X$ is normal. Let $\overline{X}$ be a normal compactification of $X$ and let $\overline{Y}$ be the integral closure of $\overline{X}$ in $Y$; this is a ramified Galois cover with Galois group $G$. We will prove the following more general claim: let $U$ be a locally closed semialgebraic subset of $\overline{X}$, and denote by $V$ its inverse image in $\overline{Y}$. Then for every element $g$ of $G$, the trace of $g$ on $R\Gamma_c((\overline{Y}^{\an},V^{\an})_{\mathrm{\acute{e}t}}, \QQ_\ell)$ lies in $\ZZ$, and it is independent of $\ell$.
If $U^{\an}$ is a compact analytic subspace of $\overline{X}^{\an}$ defined over some finite extension of $\CC(\!(t)\!)$, then $V^{\an}$ has the same properties with respect to $Y^{\an}$ by finiteness of the morphism $\overline{Y}\to \overline{X}$, and the result is a direct consequence of Berkovich's theory of \'etale cohomology with $\ZZ$-coefficients; see Theorem 7.1.1 and Corollary 7.1.2 in \cite{berk-complex}.  Now the general case follows from the same induction argument as in Lemmas 3.1 and 3.2 and Proposition 4.1 in \cite{Mar14}, using the additivity of the trace with respect to semialgebraic decompositions in $\overline{X}$. More precisely, the proofs of Lemmas 3.1 and 3.2 show that the property holds whenever $U^{\an}$ is contained in an affinoid domain inside the analytification of an affine open subscheme of $\overline{X}$, and then the proof of Proposition 4.1 yields the general result.
\end{proof}

 \begin{cor}\label{cor:multiplicative}
 Assume that $K=\Puis$. Let $f:Y\to X$ be a morphism of $K$-schemes of finite type. Let $\chi$ be an integer and let $S$ be a semialgebraic subset of $X$ such that $\eu(f^{-1}(s))=\chi$ for every $s$ in $S$. Then
$$\eu(f^{-1}(S))=\eu(S)\cdot \chi.$$
\end{cor}

\begin{proof}
 By Lemma 5.2.1 in \cite{HL15}, the set $S$ has a finite partition into locally closed semialgebraic subsets (in \cite{HL15} it is assumed that the base field $K$ is complete, but the proof remains valid for $K=\Puis$).
  Thus, we may assume that $S^{\an}$ is locally closed.  We need to show that
\begin{equation}\label{eq:eumult}
\eu((f^{\an})^{-1}(S^{\an}))=\eu(S^{\an})\cdot \chi
\end{equation}
where $f^{\an}:Y^{\an}\to X^{\an}$ is the analytification of the morphism $f$. Since the sheaves $R^if_{!}(\mathbb{F}_\ell)$ are constructible on $X$, we may assume that they are lisse on $X$ for all $i\geq 0$, by further partitioning $S$ and replacing $X$ by a suitable subscheme containing $S$.  Now equation \eqref{eq:eumult} follows from Proposition \ref{prop:eulisse} and the Leray spectral sequence with compact supports for the morphism $f^{\an}$ (see \cite[5.2.2]{berk-etale}).
  \end{proof}

\begin{cor}\label{cor:eufibers}
Assume that $K=\Puis$.  Let $f:Y\to X$ be a morphism of $K$-schemes of finite type. Let $S$ be a semialgebraic subset of $X$ and let $S_0$ be a finite subset of $S$ such that $\eu(f^{-1}(s))=0$ for every $s$ in $S\smallsetminus S_0$. Then $$\eu(f^{-1}(S))=\sum_{s\in S_0}\eu(f^{-1}(s)).$$
\end{cor}

\begin{proof}
This follows from Corollary \ref{cor:multiplicative} and additivity of Euler characteristics.
\end{proof}

\section{A geometric interpretation of the refined tropical multiplicities}\label{sec:geominterpr}
\subsection{Main conjectures}\label{ss:conjb}
We recall the set-up of \S\ref{ss:conj}. Let $\Delta$ be a lattice polygon in $\RR^2$ with $n+1$ lattice points and $g$ interior lattice points. We denote by $(Y(\Delta),L(\Delta))$ the associated polarized toric surface over the field of Puiseux series $\Puis$. The complete linear series $|L(\Delta)|$ has dimension $n$, and its general member is a smooth projective curve of genus $g$. We fix an integer $\delta$ satisfying $0\leq \delta\leq g$. Let $S$ be a set of $n-\delta$ closed points in the dense torus in $Y(\Delta)$, and let $|L| \subset |L(\Delta)|$ be the linear series of curves passing through these points. We assume that the points in the tropicalization $\trop(S)\subset \RR^2$  lie in general position. We denote by $\cC\to |L|\cong \PP^{\delta}$ the universal curve of $|L|$. In \S\ref{ss:conj}, we have conjectured the following geometric interpretations of Block and G\"ottsche's refined tropical multiplicities.

\begin{conjtag}{\ref{conj:rational}}{
 Assume that $\delta=g$ and let $\Gamma\subset \RR^2$ be a rational tropical curve of degree $\Delta$ through the points of $\trop(S)$.
Then the Block--G\"ottsche refined tropical multiplicity $N(\Gamma)$ is equal to
$ y^{-g} \chi_{-y}(\oJac(\cC_\Gamma))$.}
\end{conjtag}

\begin{conjtag}{\ref{conj:chiydelta}}{
 For any value of $\delta$ in $\{0,\ldots,g\}$, let $\Gamma\subset \RR^2$ be a tropical curve of genus $g-\delta$ and degree $\Delta$ through the points of $\trop(S)$.
Then the Block--G\"ottsche refined tropical multiplicity $N(\Gamma)$ is equal to
$ y^{-\delta} N_{\delta}(\cC_\Gamma)$.}
\end{conjtag}

\medskip

\noindent In this section, we show that Conjecture \ref{conj:chiydelta} implies Conjecture \ref{conj:rational}, and that both conjectures are true after specializing from $\chi_{-y}$ to Euler characteristic and setting $y=1$.

Recall that the definition of $N(\Gamma)$ ensures that its evaluation at $y=1$ is the classical tropical multiplicity $n_{\Gamma}$ of the tropical curve $\Gamma$. Thus, we will prove that the classical tropical curve counting multiplicities are determined by the Euler characteristics of suitable semialgebraic sets in the relative compactified Jacobian, for rational curve counting, and in the relative Hilbert schemes of points, in general. One of the key ingredients in our proof is Corollary~\ref{cor:multiplicative}, which allows us to compute the Euler characteristic of a semialgebraic family of varieties by integrating with respect to Euler characteristic on the base. As a first step, we need to show that every curve in the semialgebraic family $|L|_\Gamma$ is integral.

\subsection{Integrality of curves in $|L|_\Gamma$}  \label{sec:integral}

Let $v_1, \ldots, v_{n- \delta}$ be points in $\trop(S)\subset \RR^2$ in general position.  We recall from \cite[\S4]{Mik05} that there are only finitely many parameterized tropical curves of genus $g - \delta$ of degree $\Delta$ through $v_1, \ldots, v_{n- \delta}$, and each of these tropical curves is simple, meaning that the parameterizing curve is trivalent, the parameterization is an immersion, the image has only trivalent and 4-valent vertices, and the preimage of each 4-valent vertex has exactly two points.  Furthermore, each unbounded edge has weight 1.

\begin{prop}  \label{prop:irreducible}
Let $\Gamma$ be one of the finitely many tropical curves of degree $\Delta$ and genus $g - \delta$ through $v_1, \ldots, v_{n-\delta}$.  Then every curve in $|L|_\Gamma$ is integral and contained in the smooth locus of $Y(\Delta)$.
\end{prop}

\begin{proof}
 We have already explained in \S\ref{sec:tropicalization} that every curve with tropicalization $\Gamma$ avoids all the $0$-dimensional orbits of $Y(\Delta)$, and thus, in particular, is contained in the smooth locus of $Y(\Delta)$. Suppose $X \in |L|_\Gamma$ is not integral.  Then the associated cycle of $[X]$ decomposes nontrivially as a sum of effective cycles $[X] = [X_1] + [X_2]$.  It follows that the associated tropical cycle $[\Trop(X)]$ decomposes nontrivially as $[\Trop(X_1)] + [\Trop(X_2)]$, by \cite[Corollary~4.4.6]{OP13}.  Since the unbounded edges of $\Trop(X)$ have weight 1, the unbounded edges of $\Trop(X_1)$ and $\Trop(X_2)$ partition the edges of $\Trop(X)$ nontrivially.   We now prove that this is impossible.

Say $e$ is an edge in $\Trop(X_1)$.  We will show that $\Trop(X_2)$ has no unbounded edges, and hence is empty.  Let $v$ be a vertex of $e$.  If $v$ is trivalent in $\Trop(X)$ then $\Trop(X_1)$, being balanced, must contain the other two edges as well.  On the other hand, if $v$ is 4-valent then $\Trop(X_1)$ must contain the continuation of $e$ through $v$.
Therefore, $\Trop(X_1)$ contains the image of all edges of the parameterizing tropical curve that share a vertex with the edge parameterizing $e$.  The parameterizing curve is connected, so this means that $\Trop(X)$ is equal to $\Trop(X_1)$, set-theoretically.  It follows that $\Trop(X_1)$ contains all of the unbounded edges of $\Trop(X)$, and hence $\Trop(X_2)$ has none, as required.
\end{proof}

\subsection{Conjecture \ref{conj:chiydelta} implies Conjecture \ref{conj:rational}}
Let $\Gamma$ be a tropical curve of genus $g-\delta$ and degree $\Delta$ through the points of $\trop(S)$.  Mimicking Definition 16 in \cite{GS14}, we define the {\em motivic Hilbert zeta function} of $\Gamma$ by
 $$Z_{\Gamma}(q)=\sum_{i\geq 0}[\Hilb^i(\cC_\Gamma)]q^{i+1-g}\quad \in K_0(\VF_K)\llbracket  q \rrbracket.$$
 It has been observed by several authors \cite{kapranov,PT10,GS14} that, if we replace $\cC_\Gamma$ by an integral Gorenstein curve over a field $F$ and $K_0(\VF_K)$ by the Grothendieck ring of $F$-varieties, this zeta function shares many of the properties of the Hasse-Weil zeta function for curves over finite fields. We will now explain that this remains true for $Z_{\Gamma}(q)$, and deduce that Conjecture \ref{conj:chiydelta} implies Conjecture \ref{conj:rational}.
  We denote by $\LL$ the class of $\AA^1_K$ in $K_0(\VF_K)$.

\begin{thm}\label{thm:zeta} \item
\begin{enumerate}

\item \label{it:imply1} The product
$$f_{\Gamma}(q)=q^{g-1}(1-q)(1-q\LL)Z_{\Gamma}(q)$$ is a polynomial of degree at most $2g$ over $K_0(\VF_K)$, and satisfies the functional equation $q^{2g}\LL^gf_{\Gamma}(1/(q\LL))=f_{\Gamma}(q)$ over $K_0(\VF_K)[\LL^{-1}]$.

\item \label{it:imply2}  There exist unique elements $N^{\mathrm{mot}}_0(\cC_\Gamma),\ldots, N^{\mathrm{mot}}_g(\cC_\Gamma)$ in
 the image of the localization morphism  $K_0(\VF_K)\to K_0(\VF_K)[\LL^{-1}]$ such that
\begin{equation}\label{eq:Nmot}
Z_\Gamma(q)=\sum_{r= 0}^g N^{\mathrm{mot}}_r(\cC_{\Gamma})\left(\frac{q}{(1-q)(1-q\LL)}\right)^{r+1-g}
\end{equation}
 in $K_0(\VF_K)[\LL^{-1}]\llbracket q \rrbracket$. Moreover,
 $$\begin{array}{lll}
 N^{\mathrm{mot}}_0(\cC_\Gamma)&=&[\,|L|_{\Gamma}],
\\[2pt] N^{\mathrm{mot}}_1(\cC_\Gamma)&=&[\cC_\Gamma]+(g-1)(\LL+1)[\,|L|_{\Gamma}],
\\[2pt] N^{\mathrm{mot}}_g(\cC_\Gamma)&=&[\oJac(\cC_\Gamma)].
 \end{array}
 $$
 \end{enumerate}
\end{thm}
\begin{proof}
 We denote by $U\subset |L|$ the open subscheme parameterizing the curves in $|L|$ that are integral and do not meet the singular locus of $Y(\Delta)$. Then $|L|_{\Gamma}$ is contained in $U$, by Proposition \ref{prop:irreducible}.   We write $\cC_U\to U$ for the restriction of $\cC$ over $U$; this is a flat projective family of integral Gorenstein curves of arithmetic genus $g$. Since our family $\cC\to |L|$ has a section by construction, the compactified relative Picard schemes $\oPic^i(\cC_U)$ are all isomorphic to $\oPic^0(\cC_U)=\oJac(\cC_U)$.

 Now, we can copy the proofs of Proposition 15, Corollary 17 and Remark 18 in \cite{GS14}, using the Abel-Jacobi maps
$$\mathrm{AJ}_i:\Hilb^i(\cC_U)\to \oPic^i(\cC_U),$$ Riemann--Roch and Serre duality to prove all the properties in the statement.
 The proof of \eqref{it:imply1} is identical to that of Proposition 15 in \cite{GS14}. Since the transformation $q\mapsto q/(1-q)(1-q\LL)$ defines an automorphism of the ring $K_0(\VF_K)\llbracket q \rrbracket$, there exists a unique sequence of elements $N^{\mathrm{mot}}_0(\cC_\Gamma),N^{\mathrm{mot}}_1(\cC_\Gamma), \ldots $ in $K_0(\VF_K)$ such that
$$Z_\Gamma(q)=\sum_{r= 0}^\infty N^{\mathrm{mot}}_r(\cC_{\Gamma})\left(\frac{q}{(1-q)(1-q\LL)}\right)^{r+1-g}.$$
 Comparing the terms of degree $1-g$ and $2-g$ yields the displayed values for $N^{\mathrm{mot}}_0(\cC_\Gamma)$ and $N^{\mathrm{mot}}_1(\cC_\Gamma)$.
By the result in \eqref{it:imply1}, the series
$$P(q)=q^g\sum_{i= 1}^\infty N^{\mathrm{mot}}_{g+i}(\cC_{\Gamma})\left(\frac{q}{(1-q)(1-q\LL)}\right)^{i}$$ must be a polynomial over $K_0(\VF_K)$ of degree at most $2g$ in $q$.
  It also follows from \eqref{it:imply1} that $P(q)$ satisfies the functional equation $q^{2g}\LL^gP(1/(q\LL))=P(q)$ over $K_0(\VF_K)[\LL^{-1}]$, and this can only happen when $P(q)$ vanishes in $K_0(\VF_K)[\LL^{-1},q]$, because $P(q)$ is divisible by $q^{g+1}$. Thus $N^{\mathrm{mot}}_{g+i}$ vanishes in $K_0(\VF_K)[\LL^{-1}]$ for all $i>0$. This means that the degree $\geq g$ part of the Laurent expansion the right hand side of \eqref{eq:Nmot} only depends on the $r=g$ term. For large $i$, the Abel-Jacobi morphism $\mathrm{AJ}_i$ is a projective bundle, so that $$[\Hilb^i(\cC_\Gamma)]=[\oJac(\cC_\Gamma)][\PP^{i-g}_K].$$ It follows that $N^{\mathrm{mot}}_g(\cC_\Gamma)=[\oJac(\cC_\Gamma)]$.
\end{proof}

\begin{cor}\label{cor:imply}  The invariant $N_r(\cC_{\Gamma})$ vanishes for $r>g$, and furthermore $$N_g(\cC_{\Gamma})=\chi_{-y}(\oJac(\cC_\Gamma)).$$
In particular, Conjecture \ref{conj:chiydelta} implies Conjecture \ref{conj:rational}.
\end{cor}
\begin{proof}
By definition, $N_r(\cC_{\Gamma})=\chi_{-y}(N^{\mathrm{mot}}_r(\cC_{\Gamma}))$ for every $r\geq 0$.
\end{proof}

\subsection{Unrefined tropical multiplicities}

Here we prove one of the two main partial results toward Conjectures~\ref{conj:chiydelta} and \ref{conj:rational} mentioned in the introduction, that the conjectures are true after setting $y = 1$ and specializing from $\chi_{-y}$ to Euler characteristic.

\begin{lem} \label{lem:generalcount}
The linear series $|L|$ contains only finitely many integral curves of geometric genus $g - \delta$ that do not meet the singular points of $Y(\Delta)$, and these curves have only nodal singularities.
\end{lem}
\begin{proof}
Taking a resolution of singularities, we can reduce to the case where $Y(\Delta)$ is smooth. Let $V\subset |L(\Delta)|\cong \PP^{n}_K$ be the closure of the locus of integral curves of geometric genus $g-\delta$.
 Then the dimension of $V$ is at most $n-\delta$, so that the intersection with the general linear subspace $|L|$ of dimension $\delta$ is finite.
 By \cite[Proposition~2.1]{Har86}, the general member of each $(n - \delta)$-dimensional component of $V$ has only nodal singularities.
\end{proof}

\begin{thm} \label{thm:toricSeveri}
 We denote by $n^{\Delta,\delta}$ the toric Severi degree associated with $(\Delta,\delta)$, that is, the number of integral $\delta$-nodal curves in $|L|$.
\begin{enumerate} \item \label{it:eu1} Let $U\subset |L|$ be the open subset parameterizing integral curves that are disjoint from the singularities of $Y(\Delta)$. Then the number $n^{\Delta,\delta}$ is equal to the coefficient $n_{\delta}(\cC_U)$ in the generating series
\[
q^{1-g} \sum_{i = 0}^\infty \, \eu(\Hilb^i(\cC\times_{|L|}U)) \, q^i = \sum_{r = 0}^\infty n_r(\cC_U)\, q^{r+1-g}(1-q)^{2g-2r-2}.
\]
For $\delta=g$, we have $n^{\Delta,g}=\eu(\oJac(\cC\times_{|L|}U))$.

\item \label{it:eu2} Let $\Gamma\subset \RR^2$ be a tropical curve of genus $g-\delta$ and degree $\Delta$ through the points of $\trop(S)$.
Then the tropical multiplicity $n(\Gamma)$ of $\Gamma$ is equal to $$n_{\delta}(\cC_\Gamma):=N_{\delta}(\cC_\Gamma)\vert_{y=1}.$$
 In particular, when $g=\delta$, we have $n(\Gamma)=\eu(\oJac(\cC_\Gamma))$.
 \end{enumerate}
\end{thm}
\begin{proof}
\eqref{it:eu1}  Lemma~\ref{lem:generalcount} tells us that $\cC\times_{|L|}U \rightarrow U$ has finitely many $\delta$-nodal fibers, and all other fibers have geometric genus greater than $g - \delta$.   We can compute $\eu(\Hilb^i(\cC\times_{|L|}U ))$ by integrating with respect to Euler characteristic on the base.  Each fiber of geometric genus greater than $g - \delta$ contributes 0 to $n_{\delta}(\cC_U)$ and each $\delta$-nodal fiber contributes 1.  It follows that $n_{\delta}(\cC_U)$ equals the toric Severi degree $n^{\Delta,\delta}$. The statement for $g=\delta$ now follows from the fact that $n_g(\cC_U)=\eu(\oJac(\cC\times_{|L|}U))$ by the same arguments as in the proof of Theorem \ref{thm:zeta}.

\eqref{it:eu2} All the curves in $|L|_{\Gamma}$ are integral and contained in the smooth locus of $Y(\Delta)$, by Proposition \ref{prop:irreducible}. Thus we can copy the proof of \eqref{it:eu1}, using Corollary~\ref{cor:multiplicative} to compute $\eu(\Hilb^i(\cC_\Gamma))$ by integrating with respect to Euler characteristic on the base.
  This shows that $n_{\delta}(\cC_\Gamma)$ is the number of $\delta$-nodal fibers in $\cC_{\Gamma}\to |L|_{\Gamma}$, which is the ordinary tropical multiplicity of $\Gamma$, by the classical correspondence theorems.
 The statement for $g=\delta$ again follows from Theorem \ref{thm:zeta}, since $n_{g}(\cC_\Gamma)=\eu(N^{\mathrm{mot}}_g(\cC_\Gamma))$ by definition.
\end{proof}

\section{Refined multiplicities for genus 1}  \label{sec:genus1}

In this section, we prove Conjectures~\ref{conj:rational} and \ref{conj:chiydelta} for $g = 1$. We keep the notations from \S\ref{ss:conjb}.

\begin{thm}\label{thm:genus1}
Assume that $g=1$, and let $\delta$ be either $0$ or $1$. Let $\Gamma\subset \RR^2$ be a tropical curve of genus $g-\delta$ and degree $\Delta$ through the points of $\trop(S)$.
Then the Block--G\"ottsche refined tropical multiplicity $N(\Gamma)$ is equal to
$ y^{-\delta} N_{\delta}(\cC_\Gamma)$. If $\delta=1$, then we also have $N(\Gamma)=y^{-1}\chi_{-y}(\cC_\Gamma)$.
\end{thm}
\noindent  The case $\delta=0$ is straightforward: $\cC$ is a single elliptic curve and $N(\Gamma)=N_{0}(\cC_\Gamma)=1$. Thus, we may assume that $\delta=1$.
  Since $g=1$, the relative compactified Jacobian $\oJac(\cC)$ is simply the family $\cC$ itself. By Corollary \ref{cor:imply}, it is enough to show that
  $y^{-1}\chi_{-y}(\cC_\Gamma)$ is equal to the Block--G\"ottsche multiplicity $N(\Gamma)$ of $\Gamma$.

In the remainder of this section, we will compute $\chi_{-y}(\cC_\Gamma)$ by considering the natural embedding of $\cC$ in the toric variety $Y(\Delta) \times \PP^n$ and realizing
$\cC_\Gamma$ as the preimage of a polyhedral subset of $\Trop(\cC)$ along which all initial degenerations are smooth.  We then apply Proposition~\ref{prop:tropvol2} to compute the motivic volume of $\cC_\Gamma$ in terms of its initial degenerations, and confirm that $N(\Gamma)=y^{-1}\chi_{-y}(\cC_\Gamma)$.

\subsection{Initial degenerations of $\cC$}

Let $\PP^n \cong |L(\Delta)|$ be the projective space over $K$ with homogeneous coordinates $a_u$ for lattice points $u$ in $\Delta \cap \ZZ^2$.  The universal curve of the complete linear series $|L(\Delta)|$ is the hypersurface in $Y(\Delta) \times \PP^n$ defined by the vanishing locus of the universal equation
\begin{equation} \label{eq:universal}
f = \sum_{u \in \Delta \cap \ZZ^2} a_u x^u.
\end{equation}
Let $\overline a_u$, and $\overline x^u$ be the leading coefficients of $a_u$ and $x^u$, respectively.  We assume $g = 1$, so $\Delta$ contains a unique interior lattice point.

Let $v_1, \ldots, v_{n-1}$ be rational points in general position in $\RR^2$, let $x_i$ be a point in $T$ whose tropicalization is $v_i$, and let $|L|$ be the linear series of dimension $1$ parameterizing curves in $L(\Delta)$ that contain $x_1, \ldots, x_{n- 1}$, with $\cC \rightarrow |L|$ its universal curve.

There are finitely many parameterized tropical rational curves of degree $\Delta$ that contain $x_1, \ldots, x_{n-1}$.  Fix one such tropical curve $\Gamma$.  Recall that $|L|_\Gamma \subset |L|$ is the semialgebraic subset parameterizing curves with tropicalization $\Gamma$, and $\cC_\Gamma \rightarrow |L|_\Gamma$ is the restriction of the universal curve.

We consider four cases in our computation of $\chi_{-y} (\cC_\Gamma)$, similar to the cases in Example~\ref{ex:cubic}, according to whether $\Gamma$ contains a loop, a bounded edge of multiplicity 2 (with or without a marked point on that edge), or a vertex of multiplicity 3.  In each case, we decompose $\cC_\Gamma$ into smaller semialgebraic sets, given by preimages of faces of $\Trop(\cC_\Gamma)$, determine the contributions of preimages of different combinatorial types of faces of the tropicalization to $\chi_{-y}(\cC_\Gamma)$, and take a sum over faces to produce the desired result.

\subsection{Case 1: The tropical curve $\Gamma$ contains a loop.}  \label{sec:loop}

We observe that $\cC_\Gamma$ is the preimage in $\cC$ of a polyhedral subset of $\Trop(Y(\Delta) \times \PP^n)$ of the form $\overline \Gamma \times \mathrm{pt}$, where $\overline \Gamma$ is the closure of $\Gamma$ in $\Trop(Y(\Delta))$.  To see this, first note that there is a unique concave function $\phi: \Delta \cap \ZZ^2 \to \QQ$ whose value at the interior point is 0 and such that the corner locus of the corresponding concave piecewise linear function $\psi$ on $\RR^2$ given by
\[
\psi(v) = \min_{u \in \Delta \cap \ZZ^2} \langle u, v \rangle + \phi(u)
\]
is exactly $\Gamma$.  Since $\Gamma$ contains a loop, the interior point $0$ must be a vertex of the Newton subdivision, and it follows that all edges of $\Gamma$ have weight 1 and every lattice point in $\Delta$ is a vertex of the Newton subdivision.   Therefore, if $C \subset Y(\Delta)$ is the curve cut out by the equation
\[
f_C = \sum_{u \in \Delta \cap \ZZ^2} a_u x^u,
\]
with coefficients $a_{u} \in K$ such that $\Trop(C) = \overline \Gamma$, then each $a_{u}$ is in $K^*$.  Indeed, if we normalize so that $a_0 = 1$, then $\Trop(C \cap T)$ is equal to $\Gamma$ if and only if $\val(a_u) = \phi(u)$ for all $u$.  Therefore $\cC_\Gamma$ is the preimage in $\cC$ of $\overline \Gamma \times \mathrm{pt}$, where $\mathrm{pt}$ is the point in $\RR^n \subset \Trop(\PP^n)$ whose $u$th coordinate is $\phi(u)$.

For simplicity, we identify a face (vertex or edge) $\gamma$ of $\Gamma$ with the corresponding face of $\Gamma \times \mathrm{pt}$. We will consider the initial degenerations $\inn_w \cC$ at $\Q$-rational points $w\in \mathring{\gamma}$, and show that they are all smooth and isomorphic.  Since $\rec(\gamma)$ is a cone in the fan associated with $\Delta$ for each face $\gamma$, this will put us in position to apply Proposition~\ref{prop:tropvol3} (with $\overline{X} = \cC$) to compute $\chi_{-y}(\cC_\Gamma)$.  Indeed, with the notation from Section~\ref{sec:tropcomp}, $\cC_\Gamma$ decomposes as a disjoint union of the semialgebraic sets $\cC_\gamma$, and hence
\[
\chi_{-y}(\cC_\Gamma) = \sum_\gamma \chi_{-y}(\cC_\gamma).
\]
Moreover, Proposition~\ref{prop:tropvol3} says that 
\[
\chi_{-y}(\cC_\gamma) = (-1)^{\dim \gamma - \dim \rec(\gamma)} [\inn_\gamma \cC] / (\LL - 1)^{\dim \rec(\gamma)}.
\]

Each face of $\Gamma$ is dual to a positive dimensional face of the Newton subdivision.  Let $\gamma$ be the face of $\Gamma$ dual to $F$.  Then the initial form of the universal equation $f$ in \eqref{eq:universal} at any rational point $w$ in the interior of $\gamma$ is given by
\[
\inn_{\gamma} f = \sum_{u \in F \cap \ZZ^2} \overline{a}_{u} \overline x^u.
\]
   The linear point conditions that cut out the codimension $n-1$ linear series $|L|$ in the complete linear series $|L(\Delta)|$ involve only the coefficients $a_u$ and not the variables $x^u$. Thus their initial forms do not depend on the choice of $\gamma$ and $w$. 
  Our computations will show that $\inn_\gamma(f)$ together with the initial forms of the point conditions define a smooth closed subvariety of dimension two in the reduction of the torus $T$. Then this subvariety must be the initial degeneration of $\cC$ at $w$, for any rational point $w$ in the interior of $\gamma$, by \cite[Theorem~1.4]{OP13}. In particular, $\inn_w\cC$ does not depend on $w$; we denote it by $\inn_\gamma \cC$.
   The same observation applies in all the further cases.


\subsubsection{The linear relations imposed by point conditions} \label{sec:linearrelations}

Recall that $S$ is a set of $n-1$ points in $T(K)$ whose tropicalizations are in general position in $\RR^n$.  Say $s$ is a point in $S$ whose tropicalization lies in the edge $\gamma$ dual to the edge $[u,u']$ of the Newton subdivision.  Then the initial form of the linear relation imposed by vanishing at $s$ is simply $c\overline a_u + c' \overline a_{u'}=0$, where $c$ and $c'$ are the leading coefficients of the monomials $x^u$ and $x^{u'}$, respectively, evaluated at the point $s$.  It follows that if $u$ and $u'$ are any two vertices of the Newton subdivision connected by a series of edges that are dual to edges of $\Gamma$ containing marked points, then the linear relations force $\overline a_u$ to be a fixed nonzero scalar multiple of $\overline a_{u'}$.

We observe that the set of edges in the Newton subdivision dual to edges of $\Gamma$ that contain marked points form a disjoint union of two trees that together contain all lattice points in $\Delta$.  To see this, note that there are $n-1$ such edges among the $n+1$ lattice points in $\Delta$, and these edges cannot form a loop, due to the genericity of the marked points.  Normalizing so that $\overline a_u$ is 1 for one vertex of $\Delta$ and choosing a variable $z = \overline a_{u'}$ for some fixed $u'$ in the tree that does not contain $u$, we see that the initial forms of the linear relations determine $\overline a_{u''}$ for all lattice points in $u''$ as either a fixed element of $\CC^*$ or a fixed element of $\CC^*$ times $z$, according to which of the two trees contains $u''$.

\bigskip

We now compute  $\inn_w \cC$ and $\chi_{-y}(\cC_\gamma)$ for all faces $\gamma$ of $\Gamma$ and all rational points $w$ in $\mathring{\gamma}$.  We divide these computations into subcases, according to the combinatorial possibilities for $\gamma$.

\subsubsection{Subcase 1a:  $\gamma$ is a bounded edge of $\Gamma$.}
The face of the Newton subdivision dual to $\gamma$ is an edge of lattice length 1.  Therefore, after a change of coordinates on the dense torus in $Y(\Delta)$, the initial form $\inn_{\gamma}(f)$ is a linear function in one coordinate, with coefficients in the set $\{ \overline a_u \}$.  Since the linear relations imposed by point conditions allow us to identify each $\overline a_u$ with a monomial of degree 0 or 1 in $z$, we see that $\inn_{\gamma}\cC$ is isomorphic to a hypersurface in a three dimensional torus whose Newton polytope is an edge of length $1$.  It follows that $\inn_{\gamma} \cC \cong \mathbb{G}_{m,k}^2$, and $\chi_{-y}(\cC_\gamma) = -y^2 + 2y - 1$.

\subsubsection{Subcase 1b: $\gamma$ is an unbounded edge of $\Gamma$.}
Just as in the previous subcase, we have $\inn_{\gamma} \cC \cong \mathbb{G}_{m,k}^2$.  The only difference in this subcase is that $\dim \rec(\gamma) = 1$, and hence $\chi_{-y}(\cC_\gamma) = y - 1$.

\subsubsection{Subcase 1c: $\gamma$ is a 3-valent vertex}  After a change of coordinates, we may assume that the face dual to $\gamma$ is the standard unit triangle, so $\inn_{\gamma} f$ is a linear combination of $1$, $x$, and $y$, with coefficients in the set $\{ \overline a_u \}$.  After using the linear relations to identify each $\overline a_u$ with either a fixed element of $\CC^*$ or a fixed element of $\CC^* z$, we see that  $\inn_{\gamma}\cC$ is isomorphic to a hypersurface in the torus with coordinates $x$, $y$, and $z$, whose Newton polytope is a unimodular triangle.  It follows that $\inn_{\gamma}\cC$ is smooth, $[\inn_{\gamma}\cC] = (\LL - 1)(\LL - 2)$, and $\chi_{-y} (\cC_\gamma) = y^2 - 3y + 2$.

\subsubsection{Subcase 1d:  $\gamma$ is the 4-valent vertex $v$}  After a change of coordinates, we may assume that the face dual to $\gamma$ is the standard unit square, so $\inn_{v} f$ is a linear combination of $1$, $x$, $y$, and $xy$ with coefficients in $\CC^* \sqcup \CC^* z$.

We claim that the number of these coefficients that are in  $\CC^*$ is odd.  Let $\Gamma'$ denote the boundary between the union of the closed regions where the coefficients are in $\CC^*$ and those where the coefficients are in $\CC^* z$.  To prove the claim, we will show that $\Gamma'$ makes a turn at $v$, as shown in (a)--(d) of the following figure.  Each marked point is denoted by a $\times$ and the $4$-valent vertex $v$ is depicted as a black dot.

\bigskip

\begin {center} \scalebox{0.7}{\begin{picture}(0,0)%
\includegraphics{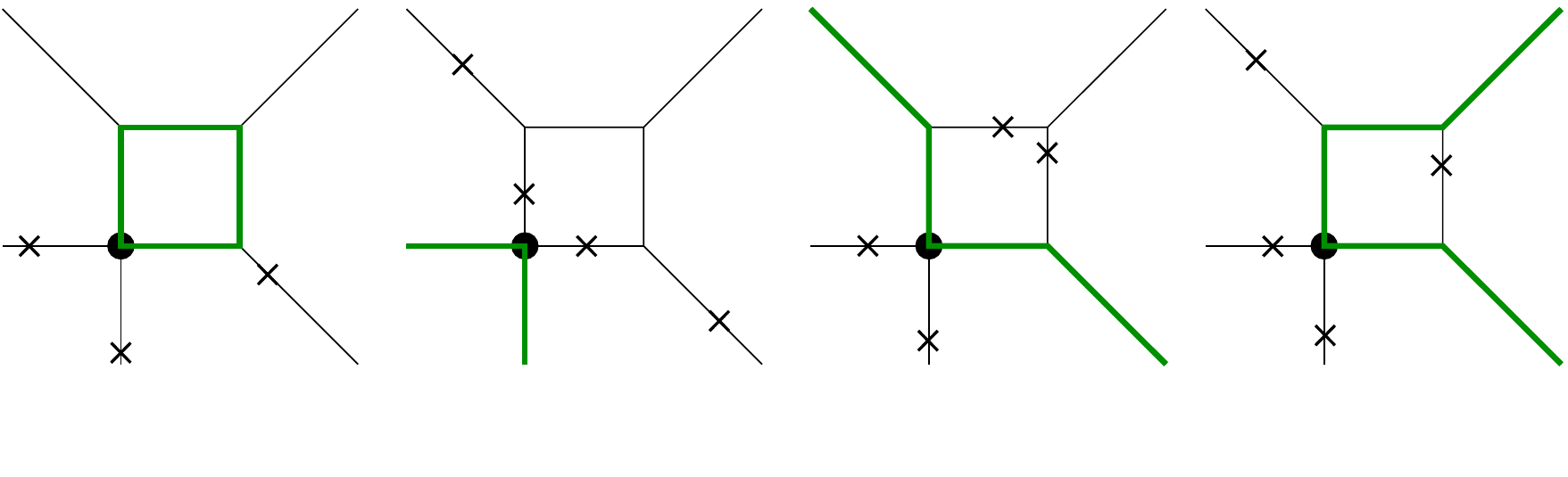}%
\end{picture}%
\setlength{\unitlength}{4144sp}%
\begingroup\makeatletter\ifx\SetFigFont\undefined%
\gdef\SetFigFont#1#2#3#4#5{%
  \reset@font\fontsize{#1}{#2pt}%
  \fontfamily{#3}\fontseries{#4}\fontshape{#5}%
  \selectfont}%
\fi\endgroup%
\begin{picture}(8032,2432)(5969,-7291)
\put(6770,-7200){\makebox(0,0)[lb]{\smash{{\SetFigFont{14}{16.8}{\familydefault}{\mddefault}{\updefault}{\color[rgb]{0,0,0}(a)}%
}}}}
\put(8776,-7171){\makebox(0,0)[lb]{\smash{{\SetFigFont{14}{16.8}{\familydefault}{\mddefault}{\updefault}{\color[rgb]{0,0,0}(b)}%
}}}}
\put(10891,-7171){\makebox(0,0)[lb]{\smash{{\SetFigFont{14}{16.8}{\familydefault}{\mddefault}{\updefault}{\color[rgb]{0,0,0}(c)}%
}}}}
\put(12961,-7171){\makebox(0,0)[lb]{\smash{{\SetFigFont{14}{16.8}{\familydefault}{\mddefault}{\updefault}{\color[rgb]{0,0,0}(d)}%
}}}}
\put(6796,-5281){\makebox(0,0)[lb]{\smash{{\SetFigFont{14}{16.8}{\familydefault}{\mddefault}{\updefault}{\color[rgb]{0,0,0}$\CC^*$}%
}}}}
\put(6041,-6532){\makebox(0,0)[lb]{\smash{{\SetFigFont{14}{16.8}{\familydefault}{\mddefault}{\updefault}{\color[rgb]{0,0,0}$\CC^* $}%
}}}}
\put(6796,-6541){\makebox(0,0)[lb]{\smash{{\SetFigFont{14}{16.8}{\familydefault}{\mddefault}{\updefault}{\color[rgb]{0,0,0}$\CC^*$}%
}}}}
\put(7426,-5911){\makebox(0,0)[lb]{\smash{{\SetFigFont{14}{16.8}{\familydefault}{\mddefault}{\updefault}{\color[rgb]{0,0,0}$\CC^*$}%
}}}}
\put(6031,-5911){\makebox(0,0)[lb]{\smash{{\SetFigFont{14}{16.8}{\familydefault}{\mddefault}{\updefault}{\color[rgb]{0,0,0}$\CC^* $}%
}}}}
\put(6751,-5911){\makebox(0,0)[lb]{\smash{{\SetFigFont{14}{16.8}{\familydefault}{\mddefault}{\updefault}{\color[rgb]{0,0,0}$\CC^* z$}%
}}}}
\put(8111,-6532){\makebox(0,0)[lb]{\smash{{\SetFigFont{14}{16.8}{\familydefault}{\mddefault}{\updefault}{\color[rgb]{0,0,0}$\CC^* z$}%
}}}}
\put(8111,-5925){\makebox(0,0)[lb]{\smash{{\SetFigFont{14}{16.8}{\familydefault}{\mddefault}{\updefault}{\color[rgb]{0,0,0}$\CC^*$}%
}}}}
\put(8866,-5281){\makebox(0,0)[lb]{\smash{{\SetFigFont{14}{16.8}{\familydefault}{\mddefault}{\updefault}{\color[rgb]{0,0,0}$\CC^*$}%
}}}}
\put(8866,-5911){\makebox(0,0)[lb]{\smash{{\SetFigFont{14}{16.8}{\familydefault}{\mddefault}{\updefault}{\color[rgb]{0,0,0}$\CC^*$}%
}}}}
\put(8866,-6541){\makebox(0,0)[lb]{\smash{{\SetFigFont{14}{16.8}{\familydefault}{\mddefault}{\updefault}{\color[rgb]{0,0,0}$\CC^*$}%
}}}}
\put(9496,-5911){\makebox(0,0)[lb]{\smash{{\SetFigFont{14}{16.8}{\familydefault}{\mddefault}{\updefault}{\color[rgb]{0,0,0}$\CC^*$}%
}}}}
\put(10181,-5925){\makebox(0,0)[lb]{\smash{{\SetFigFont{14}{16.8}{\familydefault}{\mddefault}{\updefault}{\color[rgb]{0,0,0}$\CC^* z$}%
}}}}
\put(10181,-6532){\makebox(0,0)[lb]{\smash{{\SetFigFont{14}{16.8}{\familydefault}{\mddefault}{\updefault}{\color[rgb]{0,0,0}$\CC^* z$}%
}}}}
\put(10936,-5281){\makebox(0,0)[lb]{\smash{{\SetFigFont{14}{16.8}{\familydefault}{\mddefault}{\updefault}{\color[rgb]{0,0,0}$\CC^*$}%
}}}}
\put(10936,-6541){\makebox(0,0)[lb]{\smash{{\SetFigFont{14}{16.8}{\familydefault}{\mddefault}{\updefault}{\color[rgb]{0,0,0}$\CC^* z$}%
}}}}
\put(10936,-5911){\makebox(0,0)[lb]{\smash{{\SetFigFont{14}{16.8}{\familydefault}{\mddefault}{\updefault}{\color[rgb]{0,0,0}$\CC^*$}%
}}}}
\put(11566,-5911){\makebox(0,0)[lb]{\smash{{\SetFigFont{14}{16.8}{\familydefault}{\mddefault}{\updefault}{\color[rgb]{0,0,0}$\CC^*$}%
}}}}
\put(13546,-5911){\makebox(0,0)[lb]{\smash{{\SetFigFont{14}{16.8}{\familydefault}{\mddefault}{\updefault}{\color[rgb]{0,0,0}$\CC^*$}%
}}}}
\put(12961,-5281){\makebox(0,0)[lb]{\smash{{\SetFigFont{14}{16.8}{\familydefault}{\mddefault}{\updefault}{\color[rgb]{0,0,0}$\CC^* z$}%
}}}}
\put(12961,-5911){\makebox(0,0)[lb]{\smash{{\SetFigFont{14}{16.8}{\familydefault}{\mddefault}{\updefault}{\color[rgb]{0,0,0}$\CC^*$}%
}}}}
\put(12206,-6532){\makebox(0,0)[lb]{\smash{{\SetFigFont{14}{16.8}{\familydefault}{\mddefault}{\updefault}{\color[rgb]{0,0,0}$\CC^* z$}%
}}}}
\put(12961,-6541){\makebox(0,0)[lb]{\smash{{\SetFigFont{14}{16.8}{\familydefault}{\mddefault}{\updefault}{\color[rgb]{0,0,0}$\CC^* z$}%
}}}}
\put(12196,-5911){\makebox(0,0)[lb]{\smash{{\SetFigFont{14}{16.8}{\familydefault}{\mddefault}{\updefault}{\color[rgb]{0,0,0}$\CC^* z$}%
}}}}
\end{picture}%
} \end {center}

\noindent If $\Gamma'$ is bounded, then there is only one region where the coefficient is in $\CC^* z$, and it is the bounded region, whose boundary bends at $v$ as shown in (a).  If $\Gamma'$ is unbounded then it is homeomorphic to $\RR$, with two unbounded directions.  If $\Gamma'$ does not bend at $v$, then it lifts to a \emph{string}, in the sense of \cite[Def. 3.5(a)]{GM08}, in the rational parameterizing curve $\widetilde \Gamma \rightarrow \Gamma$.  However, the rational parameterizing curve cannot contain any strings, by \cite[Remark 3.7]{GM08}.  Therefore, $\Gamma'$ bends at $v$, as shown in Figures (b)--(d), and hence the number of coefficients in $\inn_v{f}$ that are in $\CC^*$ is either 1 or 3, and the remaining coefficients are in $\CC^*z$.  It follows that $\inn_{v}f$ is a polynomial in $x$, $y$, and $z$ whose Newton polytope is a unimodular simplex of dimension 3.  We conclude that $\inn_{v} (\cC)$ is isomorphic to the intersection of a generic plane in $\PP^3$ with the dense torus.  Therefore $[\inn_v \cC] =  \LL^2 - 3\LL +3$ and $\chi_{-y}(\cC_\gamma) = y^2 - 3y + 3$.

\subsubsection{Computation of $\chi_{-y}(\cC_\Gamma)$}

We now use the computations in the four subcases above to compute $\chi_{-y}(\cC_\Gamma)$ and show that it is equal to $y$.

Say $\Delta$ has euclidean area $A/2$.  The Newton subdivision has one parallelogram of area 1, so it must contain $A-2$ unimodular triangles.  Therefore, $\Gamma$ has a unique 4-valent vertex and $A-2$ vertices that are 3-valent.  Since the union of the bounded edges has Euler characteristic zero, it follows that $\Gamma$ has $A-1$ bounded edges.  Finally, using Pick's formula and the fact that $\Delta$ has a unique interior lattice points, we see that $\Gamma$ has $A$ unbounded edges.

By the computations above, we conclude that
\[
\chi_{-y}(\cC_\Gamma) = (A - 1)(-y^2 + 2y - 1) + A(y-1) + (A-2)(y^2 - 3y+2) + y^2 - 3y + 3.
\]
Collecting terms gives $\chi_{-y}(\cC_\Gamma) = y$, as required.

\subsection{Case 2:  The tropical curve $\Gamma$ contains an edge of multiplicity 2 with a marked point.} \label{sec:withmarkedpoint}

We begin by observing that $\cC_\Gamma$ is the preimage in $\cC$ of a polyhedral subset of $\Trop(Y(\Delta) \times \PP^n)$ of the form $\overline \Gamma \times \{ \RR_{\geq 0} \cup \infty \}$.  To see this, first note that there is a unique concave function $\phi : \Delta \cap \ZZ^2 \rightarrow \QQ$ whose value at a fixed vertex $u_0$ is zero and such that the corner locus of the corresponding concave piecewise-linear function $\psi$ on $\RR^2$ given by
\[
\psi(v) = \min_{u \in \Delta \cap \ZZ^2} \langle u,v\rangle + \phi(u)
\]
is exactly $\Gamma$.  Since $\Gamma$ contains a bounded edge of multiplicity 2, the Newton subdivision must contain an edge of length 2 that contains the interior point in its relative interior.

Therefore, if $C \subset Y(\Delta)$ is the curve cut out by the equation
\[
f_C = \sum_{u \in \Delta \cap \ZZ^2} a_u x^u,
\]
with coefficients $a_{u} \in K$, then each $a_{u}$ other than the interior point must be in $K^*$ and if we normalize so that $a_{u_0} = 1$, then $\Trop(C \cap T)$ is equal to $\Gamma$ if and only if $\val(a_u) \geq \phi(u)$ for all $u$, with equality everywhere except possibly at the interior point.  In particular, $\cC_\Gamma$ is the preimage in $\cC$ of $\overline \Gamma \times \{ \RR_{\geq 0} \cup \infty \}$, where $\RR_{\geq 0} \cup \infty$ is identified with the set of points in $\Trop(\PP^n)$ whose $u$th coordinate is at least $\phi(u)$, with equality for all except the interior point.

We identify a face $\gamma$ of $\Gamma$ with the corresponding face $\gamma \times \{0\}$ in $\Trop(\cC_\Gamma)$, and we write
\[
\tilde \gamma = \gamma \times \RR_{\geq 0}.
\]
Just as in the previous case, we note that $\cC_\Gamma$ is the disjoint union, over all faces $\gamma$ of $\Gamma$ of the semialgebraic sets $\cC_\gamma \sqcup \cC_{\tilde \gamma}$.  We will show that $\inn_w \cC$ is smooth for all rational points $w$ in $\Gamma$, and that its isomorphic class is constant on the relative interiors of all the cells $\gamma$ and $\widetilde{\gamma}$.
  Then we can apply Proposition~\ref{prop:tropvol2} to compute $\chi_{-y}(\cC_\gamma)$ and $\chi_{-y}(\cC_{\tilde \gamma})$.

\subsubsection{The linear relations imposed by point conditions}  \label{sec:linearrelations2}
In this case, the edges of the Newton subdivision dual to the edges of $\Gamma$ that contain the marked points form a tree, whose vertices are all of the lattice points in $\Delta$ except the interior lattice point.

Since the edge of weight 2 contains a marked point, the tree contains the edge of lattice length 2.  In this case, the initial forms of the linear relations are different for $\gamma$ and for $\tilde \gamma$.  In $\tilde \gamma$, the coefficient of the interior lattice point vanishes in the initial form of the linear relations, which force all of the coefficients $\overline a_u$, for $u$ other than the interior point, to be fixed elements of $\CC^*$.  In $\gamma$, the coefficient of the interior lattice point does not vanish.  The linear relations force $\overline a_u$ to be a fixed nonzero scalar multiple of $\overline a_{u'}$ whenever $u$ is connected to $u'$ in this tree by a path that does not contain the edge of length 2.  In this case, we normalize so that the coefficient of the interior lattice point is 1, and the coefficients for the endpoints of the edge of length 2 are $z$ and $w$.  The linear relation imposed by the marked point on the edge of multiplicity 2 imposes a condition of the form $1 + az + bw = 0$, for some constants $a$ and $b$ in $\CC^*$.

We compute the contribution to $\chi_{-y}$ from each face of $\Gamma \times  \RR_{\geq 0}$, according to combinatorial type.  

\subsubsection{Subcase 2a: $\gamma$ is bounded edge of weight 1}  After a change of coordinates, we may assume $\inn_{\gamma}(f)$ is a linear combination of $1$ and $y$ with coefficients in $\CC^* z \sqcup \CC^* w$.  This equation, together with the linear relation $1 + az + bw$ from the point conditions, cuts out $\inn_{\gamma}(\cC)$ in the torus with coordinates $x$, $y$, $z$, and $w$.  It follows that $[\inn_\gamma (\cC)] = (\LL-1)(\LL-2)$ and $\chi_{-y}(\cC_\gamma) = -y^2 + 3y - 2$.

Similarly, we find that $\inn_{\tilde \gamma}(\cC)$ is cut out by a linear combination of $1$ and $y$ with coefficients in $\CC^*$, in the torus with coordinates $x$, $y$, and $z$.
We conclude that $[\inn_{\tilde \gamma}(\cC)] = (\LL-1)^2$ and $\chi_{-y}(\cC_{\tilde \gamma}) = -y+1$.

\subsubsection{Subcase 2b: $\gamma$ is an unbounded edge}  All unbounded edges have weight 1 and, just as for bounded edges, we find $[\inn_\gamma (\cC)] = (\LL-1)(\LL-2)$ and $[\inn_{\tilde \gamma}(\cC)] = (\LL-1)^2$.  Since $\dim \rec(\gamma) = 1$ and $\dim \rec(\tilde \gamma) = 2$, we then have $\chi_{-y}(\cC_\gamma) = y-2$ and $\chi_{-y}(\cC_{\tilde \gamma}) = 1$.

\subsubsection{Subcase 2c: $\gamma$ is a vertex that is not contained in the edge of weight 2}  After a change of coordinates, we may assume $\inn_{\gamma}(f)$ is a linear combination of $1$,  $x$, and $y$ with coefficients in $\CC^* z \sqcup \CC^* w$.  It follows that $[\inn_\gamma (\cC)] = (\LL-2)^2$ and $\chi_{-y}(\cC_\gamma) = y^2 - 4y + 4$.

Similarly, we find that $\inn_{\tilde \gamma}(\cC)$ is cut out by a linear combination of $1$, $x$, and $y$, with coefficients in $\CC^*$.  We conclude that $[\inn_{\tilde \gamma}(\cC)] = (\LL-1)(\LL-2)$ and $\chi_{-y}(\cC_{\tilde \gamma}) = y-2$.

\subsubsection{Subcase 2d: $\gamma$ is a vertex of the edge of weight 2}
After choosing coordinates, we may assume  $\inn_{\gamma} (f) = z + x + w x^2 + ay$, where $a \in \CC^* z$.  Substitute $w = (-1 -az)/b$ to get an equation in $x$, $y$, and $z$, whose Newton polytope is a pyramid over a trapezoid with height 1 and parallel edges of lengths 1 and 2.  The class of such a hypersurface is $\LL^2 - 3 \LL + 5$.  Then we need to subtract off the contribution from the locus where $z = -1/a$, which has class $\LL-2$. We conclude that $[\inn_\gamma(\cC)] = \LL^2 - 4\LL + 7$ and $\chi_{-y}(\cC_\gamma) =  y^2 - 4y + 7$.

Similarly, we find that $\inn_{\tilde \gamma} (f)$ is a linear combination of $1$, $x^2$, and $y$, with coefficients in $\CC^*$.  We conclude that $[\inn_{\tilde \gamma} (\cC)] = (\LL-1)(\LL-3)$ and $\chi_{-y}(\cC_{\tilde \gamma}) = y-3$.

\subsubsection{Subcase 2e: $\gamma$ is the edge of weight 2}  After
choosing coordinates, we may assume that
\[
\inn_{\gamma} (f) = z + x + w x^2.
\]
An explicit computation then shows that $[\inn_{\gamma} (\cC)] = \LL^2 - 6\LL + 5$ and hence $\chi_{-y}(\cC_\gamma)= -y^2 + 6y -5$.

Similarly, we find $\inn_{\tilde \gamma} (f)$ is a linear combination of $1$ and $x^2$ with coefficients in $\CC^*$.  We conclude that $[\inn_{\tilde \gamma} (\cC)] = 2 (\LL-1)^2$ and $\chi_{-y} (\cC_{\tilde \gamma}) = -2y+2$.

\subsubsection{Computation of $\chi_{-y}(\cC_\Gamma)$}

As in the previous case, Pick's formula and an Euler characteristic computation determine the number of faces of each type in terms of the euclidean area of $\Delta$.  Say $\Delta$ has euclidean area $A/2$.  Then $\Gamma$ has $A-4$ bounded edges of weight 1, $A$ unbounded edges, $A-4$ vertices that are not contained in the edge of weight 2, and $2$ vertices on that edge, in addition to the single edge of weight 2.  Applying Proposition~\ref{prop:tropvol} and combining terms for $\gamma$ and $\tilde \gamma$ in each case gives
\[
\begin{split}
\chi_{-y}(\cC_\Gamma)  =  \ &  (A-4) (-y^2 + 2y - 1) + A(y-1) + (A-4)(y^2 - 3y +2)  + \\ & +  2(y^2 - 3y +4) + (-y^2 + 4y -3).
\end{split}
\]
This simplifies to $y^2 + 2y + 1$, as required.

\subsection{Case 3: The tropical curve $\Gamma$ contains an edge of multiplicity 2 that does not contain a marked point} \label{sec:withoutmarkedpoint}

Suppose the edge of weight 2 does not contain a marked point.  Then the tree of edges in the Newton subdivision dual to edges with marked points does not contain the edge of lattice length 2.  In this case, the linear relations force all of the coefficients $\overline a_u$, for $u$ other than the interior point, to be fixed scalar multiples of each other.  In this case, we fix one of these to be $1$, and let $z$ be a variable for the coefficient of the interior lattice point.  The computations are then similar to the case above, but simpler.

\subsubsection{Subcase 3a: $\gamma$ is bounded edge of weight 1}
In this case, we find $[\inn_\gamma(\cC)]$ and $[\inn_{\tilde \gamma}(\cC)]$ are both equal to $(\LL-1)^2$.  Hence $\chi_{-y}(\cC_\gamma) = -y^2 + 2y -1$ and $\chi_{-y}(\cC_{\tilde \gamma}) = -y + 1$.

\subsubsection{Subcase 3b: $\gamma$ is an unbounded edge}
Again, $[\inn_\gamma(\cC)]$ and $[\inn_{\tilde \gamma}(\cC)]$ are both equal to $(\LL-1)^2$.  Accounting for the dimensions of the recession cones then gives $\chi_{-y}(\cC_\gamma) = y-1$ and $\chi_{-y}(\cC_{\tilde \gamma}) = 1$.

\subsubsection{Subcase 3c: $\gamma$ is a vertex that is not contained in the edge of weight 2}

In this case, we compute that $[\inn_\gamma(\cC)]$ and $[\inn_{\tilde \gamma}(\cC)]$ are both equal to $(\LL-1)(\LL-2)$.  Hence $\chi_{-y}(\cC_\gamma) = y^2 - 3y + 2$ and $\chi_{-y}(\cC_{\tilde \gamma}) = y - 2$.

\subsubsection{Subcase 3d: $\gamma$ is a vertex of the edge of weight 2}

In this case, $[\inn_\gamma(\cC)]$ and $[\inn_{\tilde \gamma}(\cC)]$ are equal to $\LL^2 - 3\LL + 4$ and $(\LL-1)(\LL-3)$, respectively.  Hence $\chi_{-y}(\cC_\gamma) = y^2 - 3y + 4$ and $\chi_{-y}(\cC_{\tilde \gamma}) = y - 3$.

\subsubsection{$\gamma$ is the edge of weight 2}

Here, we find that $[\inn_\gamma(\cC)]$ and $[\inn_{\tilde \gamma}(\cC)]$ are equal to $(\LL-3)(\LL-1)$ and $2(\LL-1)^2$, respectively.  Hence $\chi_{-y}(\cC_\gamma) = -y^2 + 4y -3$ and $\chi_{-y}(\cC_{\tilde \gamma}) = -2y + 2$.

\subsubsection{Computation of $\chi_{-y}(\cC_\Gamma)$}

The numbers of faces of each type do not depend on the location of the marked points, and hence are the same as in the previous case.  Summing over faces and combining terms for $\gamma$ and $\tilde \gamma$ produces
\[
\begin{split}
\chi_{-y}(\cC_\Gamma)  =  \ & (A - 4) (-y^2 +y) +  A y + (A-4)(y^2 - 2y) +   \\ & + 2(y^2 - 2y +1)   - y^2 + 2y - 1.
\end{split}
\]
This again simplifies to $y^2 + 2y + 1$, as required.

\subsection{Case 4: The tropical curve $\Gamma$ has a vertex of multiplicity 3}  \label{sec:specialvertex}

The computations in this case are similar to the previous two cases.  One minor difference is that some of the classes that appear as $[\inn_\gamma(\cC)]$ are not polynomials in $\LL$.  Nevertheless, the formulas for $\chi_{-y}$ are relatively simple to obtain.

We begin by observing that $\cC_\Gamma$ is again the preimage in $\cC$ of a polyhedral subset of $\Trop(Y(\Delta) \times \PP^n)$ of the form $\overline \Gamma \times \{ \RR_{\geq 0} \cup \infty \}$, and we write $\gamma$ and $\tilde \gamma$ for the faces $\gamma \times \{0\}$ and $\gamma \times \RR_{\geq 0}$ of $\Trop(\cC_\Gamma)$, respectively.

\subsubsection{The linear relations imposed by point conditions}  \label{sec:linearrelations4}
The edges of the Newton subdivision dual to edges that contain marked points form a tree on all vertices other than the interior vertex.  Therefore, the linear relations force all of the coefficients $\overline a_u$, for $u$ other than the interior point, to be fixed scalar multiples of each other.  In this case, we fix one of these to be $1$, and let $z$ be a variable for the coefficient of the interior lattice point.
\subsubsection{The initial degenerations}

In this case, $\Gamma$ has four combinatorial types of faces:  bounded edges, unbounded edges, ordinary vertices of multiplicity 1 (dual to unimodular triangles in the Newton subdivision), and one special vertex (dual to a triangle of area 3/2).

\subsubsection{Subcase 4a: $\gamma$ is a bounded edge}
If $\gamma$ is an edge, then $[\inn_\gamma(\cC)]$ and $[\inn_{\tilde \gamma}(\cC)]$ are both $(\LL-1)^2$.  When the edge is bounded, this gives $\chi_{-y}(\cC_\gamma) = -y^2 + 2y-1$ and $\chi_{-y}(\cC_{\tilde \gamma}) = -y + 1$.

\subsubsection{Subcase 4b: $\gamma$ is an unbounded edge}
Again, $[\inn_\gamma(\cC)]$ and $[\inn_{\tilde \gamma}(\cC)]$ are both $(\LL-1)^2$.  When the edge is unbounded this gives $\chi_{-y}(\cC_\gamma) = y-1$ and $\chi_{-y}(\cC_{\tilde \gamma}) = 1$.

\subsubsection{Subcase 4c: $\gamma$ is an ordinary vertex}
In this case, $[\inn_\gamma(\cC)]$ and $[\inn_{\tilde \gamma}(\cC)]$ are both equal to $(\LL-1)(\LL-2)$.  This gives $\chi_{-y}(\cC_\gamma) = y^2 - 3y +2$ and $\chi_{-y}(\cC_{\tilde \gamma}) = y - 2$.

\subsubsection{Subcase 4d: $\gamma$ is the special vertex}  In this case, $\inn_\gamma(\cC)$ is isomorphic to the complement in $\CC^* \times \CC^*$ of a smooth genus 1 curve minus 3 points.  The class of this variety is not a polynomial in $\LL$, but since $\chi_{-y}$ is additive and vanishes on smooth genus 1 curves, we find that $\chi_{-y}(\cC_\gamma) = \chi_{-y}(\inn_\gamma(\cC)) = y^2 - 2y +4$.
    Similarly, $\inn_{\tilde \gamma}(\cC)$ is isomorphic to the product of $\CC^*$ with a smooth genus 1 curve minus 3 points, and $\chi_{-y}(\cC_{\tilde \gamma}) = -3$.

\subsubsection{Computation of $\chi_{-y}(\cC_\Gamma))$}

As in the previous cases, Pick's formula tells us the number of faces of $\Gamma$ of each combinatorial type in terms of the area of $\Delta$.  In this case, if $\Delta$ has area $A/2$ then $\Gamma$ has $A-3$ bounded edges, $A$ unbounded edges, and $A-3$ ordinary vertices in addition to the 1 special vertex.  Summing over faces and combining terms for $\gamma$ and $\tilde \gamma$ then produces
\[
\begin{split}
\chi_{-y}(\cC_\Gamma) = & \  (A-3) (-y^2 +y) + Ay + (A-3) (y^2 - 2y) + y^2 - 2y + 1,
\end{split}
\]
which simplifies to give $\chi_{-y}(\cC_\Gamma) = y^2 + y + 1$, as required.  This completes the proof of Theorem~\ref{thm:genus1}.

\begin{rem}
The Block-G\"ottsche refined multiplicity associated to a tropical curve is always symmetric under the transformation $y \mapsto y^{-1}$.  Therefore, Conjecture~\ref{conj:rational} implies that $\chi_{-y} (\oJac(\cC_\Gamma))$ is invariant under the transformation $f (y) \mapsto y^g f(y^{-1})$, and we have confirmed that this is true for $g = 1$. However, in our proof of Theorem~\ref{thm:genus1}, we have expressed $\chi_{-y}(\oJac(\cC_\Gamma))$ as a sum of pieces that do not have this symmetry, and we do not know how to show that $\chi_{-y} (\oJac(\cC_\Gamma))$ has this symmetry in general.
\end{rem}

\section{Refined mulitiplicities for $\delta \leq 1$}

In this section, we extend the computations from the previous section to prove Conjecture~\ref{conj:chiydelta} for $\delta \leq 1$, in arbitrary genus.

\begin{thm}\label{thm:delta1}
Assume that $\delta$ is $0$ or $1$. Let $\Gamma\subset \RR^2$ be a tropical curve of genus $g-\delta$ and degree $\Delta$ through the points of $\trop(S)$.
Then the Block--G\"ottsche refined tropical multiplicity $N(\Gamma)$ is equal to
$ y^{-\delta} N_{\delta}(\cC_\Gamma)$.
\end{thm}

\noindent  The case $\delta=0$ is straightforward: $\cC$ is a single elliptic curve and $N(\Gamma)=N_{0}(\cC_\Gamma)=1$. Thus, we may assume that $\delta=1$.

We recall from Theorem~\ref{thm:zeta} that $N^{\mathrm{mot}}_1(\cC_\Gamma) = [\cC_\Gamma]+(g-1)(\LL+1)[\,|L|_{\Gamma}]$, and hence
\[
N_1(\cC_\Gamma) = \chi_{-y}(\cC_\Gamma) + (g-1)(y+1)(\chi_{-y}(|L|_\Gamma)).
\]
As in the case $g = 1$, we consider four cases, according to whether $\Gamma$ contains a 4-valent vertex, an edge of multiplicity 2 (with or without a marked point), or a vertex of multiplicity 3.  The computations are very similar to those in the corresponding cases for $g = 1$; the only differences are that the number of faces of each combinatorial type in $\Gamma$ depend on the genus $g$ as well as the area of $\Delta$, and we must also compute $\chi_{-y}(|L|_\Gamma)$.  (The term in our expression for $N_1(\cC_\Gamma)$ that involves $\chi_{-y}(|L|_\Gamma)$ vanishes when $g = 1$, since it appears with coefficient divisible by $(g-1)$.)

\subsection{Case 1: The tropical curve $\Gamma$ contains a 4-valent vertex}

Our computations in this case are very similar to those in Section~\ref{sec:loop}.  Every lattice point in $\Delta$ is a vertex of the Newton subdivision, and the set of edges in the Newton subdivision dual to edges of $\Gamma$ that contain marked points form a disjoint union of two trees that together contain all lattice points.  As in Section~\ref{sec:loop}, this means that every coefficient of $f$ must be nonzero, and normalizing so that one coefficient is 1 determines the valuation of all of the others.  These valuations determine a point in $\Trop(|L|)$, and the base of the family $|L|_\Gamma$ is the preimage of this point under tropicalization.  The initial degeneration of $|L|$ at this point is cut out by the initial forms of the linear relations imposed by point conditions, as discussed in Section~\ref{sec:linearrelations}.  These initial forms cut out a translate of a one parameter subgroup in the dense torus in $\PP^n$.  In particular, the initial degeneration is smooth and isomorphic to $\G_m$.  Applying Proposition~\ref{prop:tropvol} then shows $\Vol(|L|_\Gamma) = \LL - 1$ and hence $\chi_{-y}(|L|_\Gamma) = y-1$.   It remains to show that $\chi_{-y}(\cC_\Gamma) = y - (g-1)(y^2-1)$.



Just as in Section~\ref{sec:loop}, $\cC_\Gamma$ is the preimage under tropicalization of $\overline \Gamma$ and all initial degenerations $\inn_\gamma(\cC_\Gamma)$ are smooth.  Moreover, the computation of the classes of the initial degenerations in each combinatorial subcase are exactly the same.  The only remaining difference is the number of faces of each combinatorial type.

Applying Pick's formula and using the fact that $\Delta$ has $g$ interior lattice points, we find that $\Gamma$ has $A-2+g$ bounded edges, $A + 2 -2g$ unbounded edges, $A-2$ vertices of valence 3, and one vertex of valence 4.  This gives
\[
\begin{split}
\chi_{-y}(\cC_\Gamma) = \ & (A - 2+g)(-y^2 + 2y - 1) + (A+2-2g)(y-1) + \\ & + (A-2)(y^2 - 3y+2) + y^2 - 3y + 3,
\end{split}
\]
which simplifies to $y - (g-1)(y^2-1)$, as required.

\subsection{Case 2: The tropical curve $\Gamma$ contains an edge of weight 2 with a marked point}

Our computations in this case are very similar to those in Section~\ref{sec:withmarkedpoint}.  The base of the family is the preimage under tropicalization of the closure of a ray $\RR_{\geq 0}$ in $\Trop(\PP^n)$ (i.e. the projection of the space $\Gamma \times \RR_{\geq 0}$ considered in Section~\ref{sec:withmarkedpoint}).  The initial degeneration of $|L|$ at the vertex of this ray is isomorphic to the subvariety of $\G_m^2$ (with coordinates $z$ and $w$) cut out by an equation $1 + az + bw = 0$ for some nonzero scalars $a$ and $b$ (see the discussion of linear relations imposed by the point conditions in Section~\ref{sec:linearrelations2}).  In particular, the class of this degeneration is $\LL - 2$.  The initial degeneration along the interior of the ray is isomorphic to $\G_m$.  Applying Proposition~\ref{prop:tropvol2} then gives $\Vol(|L|_\Gamma) = \LL - 1$ and hence $\chi_{-y}(|L|_\Gamma) = y-1$.  It remains to show that $\chi_{-y}(\cC_\Gamma) = y^2 + 2y + 1 - (g-1)(y^2 - 1)$.

We find that $\Gamma$ contains $A - 5 + g$ bounded edges of weight 1, $A + 2 - 2g$ unbounded edges, and $A-4$ vertices that are not in the edge of weight 2, in addition to the 2 vertices on the edge of weight 2, and the edge of weight 2.  The computations of the initial degenerations in each combinatorial subcase are unchanged.  This gives
\[
\begin{split}
\chi_{-y}(\cC_\Gamma)  =  \ &  (A-5+g) (-y^2 + 2y - 1) + (A+ 2 - 2g)(y-1) + \\ &+ (A-4)(y^2 - 3y +2)  +  2(y^2 - 3y +4) + (-y^2 + 4y -3).
\end{split}
\]
This formula simplifies to $y^2 + 2y + 1 - (g-1)(y^2-1)$, as required.

\subsection{Case 3: The tropical curve $\Gamma$ contains an edge of weight 2 without a marked point}

In this case again, $|L|_\Gamma$ is the preimage under tropicalization of the closure of a ray.  However, now the initial degeneration at every point, including the vertex, is isomorphic to $\G_m$.  Applying Proposition~\ref{prop:tropvol2} gives $\Vol(|L|_\Gamma) = \LL$ and $\chi_{-y}(|L|_\Gamma) = y$.  It remains to show that $\chi_{-y}(\cC_\Gamma) = y^2 + 2y + 1 - (g-1)(y^2 + y)$.

As in the previous case $\Gamma$ contains $A -  5 + g$ bounded edges of weight 1, $A + 2 - 2g$ unbounded edges, and $A-4$ vertices that are not in the edge of weight 2, in addition to the 2 vertices on the edge of weight 2, and the edge of weight 2.  The initial degenerations in each combinatorial subcase are just as in Section~\ref{sec:withoutmarkedpoint}.  This gives
\[
\begin{split}
\chi_{-y}(\cC_\Gamma)  =  \ & (A - 5 + g) (-y^2 +y) +  (A + 2 -2g) y + (A-4)(y^2 - 2y) +   \\ & + 2(y^2 - 2y +1)   - y^2 + 2y - 1,
\end{split}
\]
which simplifies to $y^2+ 2y + 1 - (g-1)(y^2 + y)$, as required.

\subsection{Case 4: The tropical curve $\Gamma$ contains a vertex of multiplicity 3}

In this case, a computation identical to that in the previous case shows $\chi_{-y}(|L|_\Gamma) = y$.  It remains to show that $\chi_{-y}(\cC_\Gamma) = y^2 + y + 1 - (g-1)(y^2 + y)$.

We find that $\Gamma$ contains $A - 4 + g$ bounded edges, $A+ 2 - 2g$ unbounded edges, $A-3$ ordinary vertices, and the 1 special vertex.  The initial degenerations are exactly as in Section~\ref{sec:specialvertex}.  This gives
\[
\chi_{-y}(\cC_\Gamma) =  (A-4+g) (-y^2 +y) + (A+2-2g)y + (A-3) (y^2 - 2y) + y^2 - 2y + 1,
\]
which simplifies to $y^2+y+1 - (g-1)(y^2+y)$, as required.  This completes the proof of Theorem~\ref{thm:delta1}.

\bibliographystyle{amsalpha}
\bibliography{RefinedCurveCounting}

\end{document}